\documentclass[12pt]{amsart}
\usepackage{pgf,tikz,pgfplots} 
\pgfplotsset{compat=1.14}
\usepackage{mathrsfs}
\usetikzlibrary{arrows}
\usetikzlibrary{patterns}
\usepackage{pgfplots}
\usetikzlibrary{intersections, pgfplots.fillbetween}
\usepackage[colorlinks=true,urlcolor=blue, citecolor=red,linkcolor=blue,linktocpage,pdfpagelabels, bookmarksnumbered,bookmarksopen]{hyperref}
\usepackage[hyperpageref]{backref}

\usepackage{cleveref}
\usepackage{xcolor}
\usepackage{amsthm} 
\usepackage{latexsym,amsmath,amssymb}

\usepackage{accents}
\usepackage[colorinlistoftodos,prependcaption,textsize=tiny]{todonotes}
\usepackage{a4wide}
\usepackage{soul}
\usepackage{mathtools} 
\usepackage{xparse} 

\pgfdeclarelayer{ft}
\pgfdeclarelayer{bg}
\pgfsetlayers{bg,main,ft}

\title{H\"older maps under Pfaffian constraints}

\author{Armin Schikorra}
\address[Armin Schikorra]{Department of Mathematics,
University of Pittsburgh,
301 Thackeray Hall,
Pittsburgh, PA 15260, USA \and Department of Mathematics,
Chulalongkorn University,
Bangkok,
Thailand}
\email{armin@pitt.edu}

\definecolor{indigo}{rgb}{0.29, 0.0, 0.51}
\definecolor{p1}{gray}{0.4}
\definecolor{p2}{gray}{0.6}
\definecolor{p3}{gray}{0.98}
\definecolor{p4}{gray}{0.8}
\definecolor{p5}{gray}{0.9}
plus 6pt minus 12pt
\def\eps{\varepsilon}

\def\B{\mathbb{B}}

\def\N{{\mathbb N}}

\def\H{{\mathbb H}}

\def\S{{\mathbb{S}}}

\newtheorem{theorem}{Theorem}
\newtheorem{lemma}[theorem]{Lemma}
\newtheorem{corollary}[theorem]{Corollary}
\newtheorem{proposition}[theorem]{Proposition}

\newtheorem{definition}[theorem]{Definition}

\newtheorem{conjecture}[theorem]{Conjecture}

\newcommand{\loc}{\mathrm{loc}}

\def\dist{{\rm dist\,}}

\def\lip{{\rm Lip\,}}
\def\rank{{\rm rank\,}}
\def\supp{{\rm supp\,}}
\def\loc{{\rm loc}}

\newcommand{\R}{\mathbb{R}}

\newcommand{\Z}{\mathbb{Z}}

\newcommand{\brac}[1]{\left (#1 \right )}
\newcommand{\abs}[1]{\left |#1 \right |}
\newcommand{\Ep}{\bigwedge\nolimits}

\newcommand{\barint}{
\rule[.036in]{.12in}{.009in}\kern-.16in \displaystyle\int }

\newcommand{\barcal}{\mbox{$ \rule[.036in]{.11in}{.007in}\kern-.128in\int $}}

\def\mvint_#1{\mathchoice
          {\mathop{\vrule width 6pt height 3 pt depth -2.5pt
                  \kern -8pt \intop}\nolimits_{\kern -3pt #1}}%
          {\mathop{\vrule width 5pt height 3 pt depth -2.6pt
                  \kern -6pt \intop}\nolimits_{#1}}%
          {\mathop{\vrule width 5pt height 3 pt depth -2.6pt
                  \kern -6pt \intop}\nolimits_{#1}}%
          {\mathop{\vrule width 5pt height 3 pt depth -2.6pt
                  \kern -6pt \intop}\nolimits_{#1}}}

\numberwithin{theorem}{section} \numberwithin{equation}{section}

\newcommand{\aleq}{\precsim}

\newcommand{\aeq}{\asymp}

\let\latexchi\chi
\makeatletter
\renewcommand\chi{\@ifnextchar_\sub@chi\latexchi}
\newcommand{\sub@chi}[2]{
  \@ifnextchar^{\subsup@chi{#2}}{\latexchi_{#2}}%
}
\newcommand{\subsup@chi}[3]{
  \latexchi_{#1}^{#3}%
}
\makeatother

\makeatletter
\def\tikz@arc@opt[#1]{
  {%
    \tikzset{every arc/.try,#1}%
    \pgfkeysgetvalue{/tikz/start angle}\tikz@s
    \pgfkeysgetvalue{/tikz/end angle}\tikz@e
    \pgfkeysgetvalue{/tikz/delta angle}\tikz@d
    \ifx\tikz@s\pgfutil@empty%
      \pgfmathsetmacro\tikz@s{\tikz@e-\tikz@d}
    \else
      \ifx\tikz@e\pgfutil@empty%
        \pgfmathsetmacro\tikz@e{\tikz@s+\tikz@d}
      \fi%
    \fi
    \tikz@arc@moveto
    \xdef\pgf@marshal{\noexpand%
    \tikz@do@arc{\tikz@s}{\tikz@e}
      {\pgfkeysvalueof{/tikz/x radius}}
      {\pgfkeysvalueof{/tikz/y radius}}}%
  }%
  \pgf@marshal%
  \tikz@arcfinal%
}
\let\tikz@arc@moveto\relax
\def\tikz@arc@movetolineto#1{%
  \def\tikz@arc@moveto{\tikz@@@parse@polar{\tikz@arc@@movetolineto#1}(\tikz@s:\pgfkeysvalueof{/tikz/x radius} and \pgfkeysvalueof{/tikz/y radius})}}
\def\tikz@arc@@movetolineto#1#2{#1{\pgfpointadd{#2}{\tikz@last@position@saved}}}
\tikzset{%
  move to start/.code=\tikz@arc@movetolineto\pgfpathmoveto,%
  line to start/.code=\tikz@arc@movetolineto\pgfpathlineto}
\makeatother

\begin{document}
\begin{abstract}
Given a one form $\lambda$ in $\mathbb{R}^N$ and $f: \mathbb{S}^{n} \to \mathbb{R}^N$ with $f^\ast \lambda = 0$ we discuss the maximal H\"older regularity of extensions $F: \mathbb{B}^{n+1} \to \mathbb{R}^N$ such that $F^\ast\lambda = 0$ in distributional sense.

Our analysis applies to mappings into the Heisenberg groups $\mathbb{H}_n$. It implies in particular that for all $n \geq 1$ any smooth horizontal map $f: \mathbb{S}^{n} \to \mathbb{H}_n$ can be extended to a $C^\alpha$-map $F: \mathbb{B}^{n+1} \to \mathbb{H}_n$ for some $\alpha > 1/2$. Moreover, if $n \geq 3$ we find $C^\alpha$-embeddings from $\mathbb{B}^{n+1}$ into $\mathbb{H}_n$ for some $\alpha > \frac{1}{2}$.

{In the appendix we discuss an (as of now unverified) approach to extend these arguments to find $C^\alpha$-embeddings from $\mathbb{B}^{2}$ into $\mathbb{H}_1$ for some $\alpha > \frac{1}{2}$, assisted by GPT-6 Astra.}
\end{abstract}
\dedicatory{Dedicated to Piotr Haj\l{}asz on the occasion of his 60th birthday}
\maketitle
%

\section{Introduction}
The main motivation for the present work is the following conjecture originating from and unresolved since Gromov's seminal work \cite{Gromov96}.
\begin{conjecture}\label{conj:gromov}
Let $\alpha \in (0,1)$ and let $k > n$. Every $\alpha$-H\"older continuous embedding
$F: \B^k \subset \R^k \to \mathbb{H}_n$ satisfies $\alpha \leq \frac{1}{2}$.
\end{conjecture}
Gromov proved the bound $\alpha \leq \frac{k}{k+1}$, \cite[Corollary 3.1.A]{Gromov96}, and it is easy to show that the bound cannot be $<\frac{1}{2}$ -- and that is the best known result until today for this question, despite several attempts to the positive and negative for this conjecture.

The Heisenberg group $\mathbb H_n$ can be identified with $\R^{2n+1}$, with coordinates
\[
        (x,y,t)\in \R^n\times \R^n\times \R,
\]
and equipped with the group law
\[
        (x,y,t)\ast (x',y',t')
        =
        \left(
        x+x',
        y+y',
        t+t'
        -\frac12\sum_{j=1}^n (x_jy'_j-y_jx'_j)
        \right).
\]
The horizontal distribution of $\mathbb H_n$ is $H\mathbb H_n:=\ker \omega_{\H_n}$
where $\omega_{\H_n}$ denotes the contact form
\begin{equation}\label{eq:Hcontactform}
        \omega_{\H_n}
        =
        dt+\frac12\sum_{j=1}^n
        \bigl(x_j\,dy_j-y_j\,dx_j\bigr).
\end{equation}

Equivalently, the horizontal space $H\mathbb H_n$ is spanned by the vector fields
\[
        X_j
        =
        \partial_{x_j}+\frac{y_j}{2}\partial_t,
        \qquad
        Y_j
        =
        \partial_{y_j}-\frac{x_j}{2}\partial_t,
        \qquad j=1,\ldots,n.
\]
A curve $\gamma(s)=(x(s),y(s),t(s))$ is called horizontal if $\gamma^\ast \omega_{\H_n}=0$.
The Carnot--Carath\'eodory distance between two points $p,q\in \mathbb H_n$ is the shortest length of a horizontal path from $p$ to
$q$. (Smooth) motion in the $t$-direction
is possible only indirectly, by moving around loops in the horizontal
variables. This is why the $x,y$ directions behave like Euclidean directions,
whereas the $t$-direction has the different, fractal, scaling.
A locally equivalent metric is  the Koranyi metric
\[
        d_{\mathcal K}(p,q) =
        \|q^{-1}\ast p\|_{\mathcal K}
        =
        \left(
        \left(|x-x'|^2+|y-y'|^2\right)^2
        +
        16\left(
        t-t'
        +\frac12\sum_{j=1}^n (x'_jy_j-y'_jx_j)
        \right)^2
        \right)^{1/4}.
\]
The systematic study of the local geometry of the Heisenberg groups $\H_n$ goes back to the aforementioned work by Gromov, \cite{Gromov96}. He argued that due to the fractal nature of the Heisenberg groups (and more general Carnot-Caratheorody groups) H\"older mappings are the natural category, and from his work originates \Cref{conj:gromov}, cf. \cite[Theorem 1.10.f]{BHW14}.

It is tempting to consider $\H_n$ as a metric space with a group structure and try to use that group structure to attack questions such as \Cref{conj:gromov} -- we in particular highlight the recent beautiful work \cite{WY} in this spirit.

Here, we rather take a different point of view, we will drop the metric and group structure of $\H_n$ and instead consider $\H_n$ as a Euclidean space $\R^{2n+1}$ with a \emph{differential restriction}:

It is well-known that a Lipschitz map $f: \Omega \subset \R^k \to \R^{2n+1}$ is Lipschitz as a map between metric spaces $f: \Omega \to \H_n$ if and only if
\begin{equation}\label{eq:pullbackdefhorizontal}
 f^\ast \omega_{\H^n} = 0 \quad \text{a.e. in $\Omega$}
\end{equation}
where $\omega_{\H_n}$ is the contact structure \eqref{eq:Hcontactform}. This differential equation is a great tool to study the behavior of Lipschitz maps into the Heisenberg group, as it implies that $\rank Df \leq n$ a.e. in $\Omega$ -- and this can be used regarding contractability and Lipschitz homotopy groups \cite{BHW14,WengerY1,WengerY2,Balogh-Faessler-2009,HST14,H18}.

For a H\"older map, the a.e. notion of \eqref{eq:pullbackdefhorizontal} makes no geometric sense, even for maps for which the pointwise a.e. notion can be defined, cf. \cite{Salem1943,ParkSchikorra25}. Essentially this is equivalent to the observation that for $f: [0,1] \to \R$ a notion of $f'=0$ a.e. does not imply that $f$ is constant (take the Heaviside function, or the Cantor staircase). Thus, the H\"older theory began from a different analytic viewpoint.  Le Donne and
Z\"ust used H\"older Jacobians and integration of H\"older forms in the Heisenberg setting \cite{LZ}, building on Z\"ust's theory of H\"older currents \cite{Z11}.
Balogh, Kozhevnikov and Pansu adapted these ideas to prove restrictions on H\"older embeddings from Euclidean spaces to Carnot groups \cite{BKP}. On the other hand, the aforementioned work Wenger--Young, and Lytchak \cite{WY,LWY20} proved for $n=1$ that for any $\alpha <\frac{k}{k+1}$, $k=2,3$ there are $C^\alpha$-extensions of maps $f: \S^k \to \H_1$ (not embeddings). Their argument is based on the metric structure and the group law of $\H_1$, it seems to be one-dimensional in nature and does not easily generalize to mappings into $\H_n$.

Here, we replace the notion \eqref{eq:pullbackdefhorizontal} with
\begin{equation}\label{eq:pullbackheisenberg}
 f^\ast\omega_{\H^n} = 0 \quad \text{in distributional sense}.
\end{equation}
See \Cref{def:distributionalsensepullback} for the precise notion. This tiny change from ``a.e.'' to ``distributional'' makes all the difference. Indeed any Euclidean $f \in C^\alpha(\B^k,\R^{2n+1})$ belongs to $f \in C^\alpha(\B^k,\H_n)$ if and only if \eqref{eq:pullbackheisenberg} holds.

Thus, we stop for the moment to discuss the Heisenberg group, but instead discuss

\subsection*{Maps between Euclidean spaces with a Pfaffian form restriction} Fix a non-vanishing Pfaffian form in $\R^N$, $\lambda \in C^\infty(\Ep^1 \R^N)$. Motivated from the questions above, we are interested in studying how ``difficult'' it is to find H\"older maps $f$ that satisfy $f^\ast \lambda=0$. More precisely, we want to discuss two aspects:
\begin{itemize}
 \item Given $f: \partial \B^n \to \R^N$ with $f^\ast \lambda = 0$, can we extend $f$ to $F: \overline{\B^n} \to \R^N$ with $F^\ast \lambda = 0$?
        \item Given $F_0: \B^n \to \R^N$, do we find a ``close-by'' map $\tilde{F}$ which satisfies $\tilde{F}^\ast \lambda=0$?
\end{itemize}

Here, an henceforth, by $f^\ast \lambda = 0$ we mean the geometrically relevant distributional sense. Since $\lambda$ is a one-form, this restricts our interest to $C^\alpha$-maps, $\alpha > \frac{1}{2}$ where the distributional notion is well-defined, see \Cref{def:distributionalsensepullback}.

The answers to above questions change depending on the H\"older regularity of the maps involved, and they depend on the\footnote{This definition is related to the \emph{Pfaff rank}, which for $\lambda$ with the properties above is $2d+1$.
} ``rank'' of $\lambda$:

\begin{definition}
Let $U \subset \R^N$ open, $\lambda \in C^\infty(U;\Ep^1 \R^N)$ be a nowhere vanishing one form. We say that $\lambda$ has rank $d \in \{0,1,\ldots \lfloor \frac{N-1}{2}\rfloor \}$ if
\[
        \lambda\wedge(d\lambda)^d\neq0,
        \qquad \text{and} \qquad
        \lambda\wedge(d\lambda)^{d+1}\equiv 0 \quad \text{in $U$}
\]
\end{definition}

First we consider the $d=0$ case. The following is very simple to prove.
\begin{proposition}[Lipschitz Extension for $d=0$]\label{pr:Frobeniussituation}
Let $n\ge 2$. Suppose that for an open set $U \subset \R^N$
\[
        \lambda\wedge d\lambda\equiv 0 \quad \text{but $\lambda \neq 0$} \quad \text{in $U$}
\]
Then for any $X_0 \in U$ there exists a smaller neighborhood $V\Subset U$ of $X_0$ such that every $C^\infty$ map
\[
        f:\partial \B^n\to V,
        \qquad
        f^\ast\lambda=0,
\]
has a Lipschitz extension $F:\B^n\to V$ such that
\[
        F=f \quad\text{on }\partial \B^n,
        \qquad
        F^\ast\lambda=0 \quad \text{in $\B^n$}
\]
\end{proposition}

If the rank of $\lambda$ is positive, things become more interesting, and obstructions to extensions appear. In \cite{HajlaszMirraSchikorra,HS2023,S20}, and likely in earlier works, the following is \emph{de facto} established:
\begin{proposition}[Non-Extension]\label{pr:nonextension}
Suppose that for an open set $U \subset \R^N$ there exists
$d\in\{1,\ldots,\lfloor (N-1)/2\rfloor\}$ such that $\lambda$ has rank $d$ in $U$.

Then for any $X_0 \in U$ there exists a small contractible neighborhood $V \subset U$, $V \ni X_0$ and $f: \partial \B^{d+1} \to V$ a smooth embedding with $f^\ast \lambda = 0$ such that for any $\alpha > \frac{d+1}{d+2}$ it is impossible to find a $C^\alpha$-extension $F: \overline{\B^{d+1}} \to V$, i.e. there is no $F \in C^{\alpha}(\overline{\B^{d+1}},\R^N)$, $F \Big |_{\partial \B^{d+1}} = f$ and $F^\ast \lambda = 0$ in $\B^{d+1}$ in distributional sense.
\end{proposition}

Since $\lambda$ is a Pfaffian form and the distributional $f^\ast \lambda$ makes sense as long as $f \in C^{\alpha}$, $\alpha > \frac{1}{2}$, one could believe that \Cref{pr:nonextension} is true for any $\alpha > \frac{1}{2}$, but that is not the case. This is our first main result:
\begin{theorem}
\label{th:lambdaext}
Let $n \geq 2$, $U\subset\R^N$ be open, let $\lambda\in C^\infty(\Lambda^1\R^N)$ be nowhere zero and suppose that $\lambda$ has rank $d$ in $U$ for some $d \geq 1$. Set $s:=\left\lceil\frac{n+1}{d}\right\rceil$ and let $X_0 \in U$.

Then there exists a smaller neighborhood $V\Subset U$ of $X_0$ such that for every
\[
        f \in C^\infty(\partial \B^n,V),
        \qquad
        f^\ast\lambda=0 \text{ on $\partial \B^n$}
\]
and for any
$\frac{1}{2}<\alpha<\frac{s+1}{2s+1},
$ there exists an extension
$F\in C^\alpha(\overline{\B^n};V)$, $F \Big |_{\partial \B^n} = f$ such that
        \[
        F^\ast\lambda=0 \quad \text{in distributional sense in $\overline{\B^n}$}.
\]

\end{theorem}

Next we adress the question about genericity of the condition $F^\ast \lambda = 0$ -- i.e. given a map $F_0$, can we find a map $F$ closeby so that $F^\ast \lambda = 0$? Again we first discuss the obstruction:

\begin{proposition}\label{pr:nonclose}
Suppose that for an open set $U \subset \R^N$ there exists
$d\in\{1,\ldots,\lfloor (N-1)/2\rfloor\}$ such that $\lambda$ has rank $d$ in $U$.

Then for any $X_0 \in U$ there exists a small neighborhood $V \subset U$, $V \ni X_0$ and $F_0 \in C^\infty(\overline{\B^{d+1}},V)$ with the following property. For any $\alpha > \frac{d+1}{d+2}$ there exists an $\eps > 0$ any map $F: \B^{d+1} \to V$ with $\|F-F_0\|_{C^\alpha(\B^{d+1})} \leq \eps$ does \emph{not} satisfy $F^\ast \lambda =0$ in $\B^{d+1}$
\end{proposition}

Similar to \Cref{th:lambdaext} this obstruction dissappears for $\alpha \approx \frac{1}{2}$, our second main result.
Notably, we find that we can keep certain coordinates of the original map $F_0$ intact -- at the expense of lowering the regularity of $F$, this will be crucial to construct embeddings.
\begin{theorem}\label{th:closebymaps}
Let $N\ge 2d+1$, $d\ge1$, $n\ge2$, let $U\subset\R^N$ be open, and assume $\lambda\in C^\infty(\Lambda^1\R^N)$ has rank $d$ in $U$ and take $X_0 \in U$.

Then there exists a smaller neighborhood $V\Subset U$ of $X_0$ and a diffeomorphism $\Phi: V \to \Phi(V) \subset \R^N$ with the following properties:

Take $b \in \{1,\ldots,d\}$, and set $s:=\left\lceil\frac{n+1}{b}\right\rceil$ and take any
\[        \frac12<\alpha<\frac{s+1}{2s+1}.
\]
and $\beta \in (0,\frac{1}{2})$.

For any $F_0\in C^\infty(\overline{\B^n};V)$ and for any open $\Omega \Subset \B^n$ there exists $F\in C^\alpha(\B^n;V)$ such that
\[        F^\ast\lambda=0 \quad \text{in distributional sense in $\Omega$},
\]
and
\[        \|F-F_0\|_{C^\beta(\B^n)}<\varepsilon .
\]

Moreover we can preserve coefficients in the following sense: We have
\[       (\Phi \circ F)^{\ell} = (\Phi \circ F_0)^\ell \quad \forall \ell \in \{2b+2,\ldots N\} \quad \text{in $\B^n$}
\]

\end{theorem}

Both, \Cref{th:lambdaext} and \Cref{th:closebymaps} are proven using ideas from Convex Integration. Since we are interested in H\"older continuity one might believe that the right techniques are staircase laminates as in \cite{faracomilton, AFS08}. Indeed, combined with in-approximations these can be used to construct homeomorphisms \cite{FMCO18}. However we encountered several obstacles to that strategy: laminates produce Sobolev maps that are $C^\alpha$ for all $\alpha <1$, it is unclear how to differentiate between different H\"older exponents. Also observe $f^\ast \lambda = 0$ is a mixed derivative condition which makes application of ``pure'' staircase laminates challenging. Lastly, laminates tend to produce geometrically meaningless ``a.e.'' Jacobian conditions, not distributional solutions, as was discussed in \cite{ParkSchikorra25} with respect to homeomorphisms with pointwise a.e. vanishing Jacobian vs. distributional vanishing Jacobians. So instead of using a laminate approach, we decided to go back to and adapt the Nash-Kuiper corrugations for the isometric embedding problem.

\subsection*{Applications to H\"older maps into the Heisenberg group}
Firstly, we obtain the following extension theorem reminiscent of results by Wenger and Young \cite{WY} for $n=1$. We stress once more that they are using crucially the \emph{metric curve} (or group-) structure of the Heisenberg group, but -- for $n =1$ -- they obtain a better exponent than we do: $\alpha <\frac{2}{3}$. However, it is unclear how to extend their argument to larger dimensions $n \geq 2$. Our Pfaffian form-based argument is completely different, less geometric and more analytic, and thus more flexible -- it works in any dimension without substantial change. It is also likely adaptable to more general spaces than the Heisenberg group, as long as the geometry can be properly described in differential forms.

\begin{corollary}[H\"older extension in $\mathbb H_n$ ]
\label{co:Heisenbergextension}\label{cor:heisenberghomeo}
Let $n\ge1$ and $k\ge2$. Set $s:=\left\lceil\frac{k+1}{n}\right\rceil$.

For any
\begin{equation}\label{eq:heisenberghomeo-alpha}
        \frac12<\alpha<\frac{s+1}{2s+1}
\end{equation}
and any $f \in \lip(\partial\B^k;\mathbb H_n) \cap C^\infty(\partial \B^k;\R^{2n+1})$ there exists $F\in C^\alpha(\overline{\B^k};\H_n)$
with $F=f$ on $\partial\B^k$.

In particular if $k=n+1$, we can choose any $\alpha \in (\frac{1}{2}, \frac{4}{7})$ if $n=1$ and $\alpha \in (\frac{1}{2},\frac{3}{5})$ if $n \geq 2$.
\end{corollary}

The following then disproves \Cref{conj:gromov}
\begin{corollary}[Embeddings into $\mathbb H_n$ for $n \geq 3$]
\label{cor:heisenberghomeov1}
For any $n\geq 3$ there exists $\alpha > \frac{1}{2}$ and an embedding $F:\B^{n+1}\to \R^{2n+1}$ such that
\[
        F\in C^\alpha_{\loc}(\B^{n+1};\mathbb H_n).
\]
\end{corollary}

\subsection*{Outline}
In \Cref{s:prelims} we gather the basic preliminaries, define the basics of the distributional pullbacks, and prove the propositions at the beginning of the introduction. In \Cref{s:nash} we prepare the main steps of our Nash-corrugation type argument, i.e. introduce the update machinery for the main theorems. The main theorems are then proven in \Cref{s:extension} and \Cref{sec:homeomorphic-graph-pfaff-nash}. The Corollaries for the Heisenberg group are established in
\Cref{s:applicationstoheisenberg}.
\subsection*{Notation}
Throughout the paper $\B^n=B(0,1) \subset \R^n$ denotes the open unit ball. For two nonnegative numbers $A$ and $B$ the notion $A \aleq B$ means there is a positive constant $C>0$ whose dependency should be clear from the context such that $A \leq C B$. $A \aeq B$ means $A \aleq B$ and $B \aleq A$.

\subsection*{Acknowledgement}
A.S. is funded by NSF Career DMS-2044898. This work was partially supported by the Simons Foundation grant (award no. SFI-MPS-T-Institutes-00010825) and from State Treasury funds as part of a task commissioned by the Minister of Science and Higher Education under the project “Organization of the Simons Semesters at the Banach Center - New Energies in 2026-2028” (agreement no. MNiSW/2025/DAP/491).”

Discussions with Behnam Esmayli are gratefully acknowledged.

Part of the work leading to this article is assisted by chatgpt. All mathematical validation, final proof decisions, and final wording remain the sole responsibility of the human author.

\section{Preliminaries, Constant Rank theorem and Proofs of Propositions~\ref{pr:Frobeniussituation}, \ref{pr:nonextension}, \ref{pr:nonclose}}
\subsection{Distributional pullback of forms}\label{s:prelims}
The distributional interpretation of Jacobians, and thus of forms is well-known to experts. Its study has a long tradition pioneered among others by fundamental works of Ball \cite{B76}, Brezis, Nirenberg \cite{BN95}, Coifman, Lions \cite{CLMS}, M\"uller \cite{M90}, Reshetnyak \cite{R68}, Wente \cite{W69}, Tartar \cite{T79}. The necessary estimates can be proven generally with Littlewood-Paley theory \cite{WY99a}, and more elegantly with extension methods, cf. \cite{Coifman-Jones-Semmes-1989,C91,LS18}. For adaptations to H\"older maps see also \cite{HajlaszMirraSchikorra}. In particular we highlight Brezis-Nguyen's argument in \cite{Brezis-Nguyen-2011}.

\begin{lemma}\label{la:smoothapproxpullback}
Let $\omega \in C^\infty(\Ep^k \R^N)$, $K \subset \R^n$ is either a fixed bounded $C^\infty$ domain, or the closure of such a domain, or the boundary of a smooth domain, $\alpha > \frac{k}{k+1}$ and $\Lambda > 0$. There exists a constant $C= C(K,\omega,\Lambda,\alpha)$ such that the following holds
for any $f,g \in C^\infty(K,\R^N)$ with $\|f\|_{C^\alpha(K)}, \|g\|_{C^\alpha(K)} \leq \Lambda$:

For any $\varphi \in C_c^\infty(\Ep^{\dim(K)-k} K)$
\[
 \abs{\int_{K} \brac{f^\ast\omega-g^\ast\omega} \wedge \varphi} \leq C\, \|f-g\|_{C^{\alpha}(K)}\, \|\varphi\|_{\lip}
\]
and if $\alpha > \frac{k+1}{k+2}$
\[
 \abs{\int_{K} \brac{f^\ast\omega-g^\ast\omega} \wedge \varphi} \leq C\, \|f-g\|_{C^{\alpha}(K)}\, \|\varphi\|_{C^\alpha}
\]
\end{lemma}
The above motivates the
\begin{definition}[Distributional pullbacks]\label{def:distributionalsensepullback}
Let $\omega \in C^\infty(\Ep^1 \R^N)$ and $f \in C^\alpha(K,\R^N)$ where $K$ is either an open set or the closure of an open set in $\R^n$.

Let $f_\eps \in C^\infty(K,\R^N)$ be any smooth approximation of $f$ as $\eps\to0$ in the sense of $C^{\tilde{\alpha}}(K,\R^N)$ for some $\tilde{\alpha} \in (\frac{1}{2},\alpha)$. Then we set for $\varphi \in C_c^\infty(\Ep^{\dim(K)-1} K)$
\[
 f^\ast\omega[\varphi] := \lim_{\eps \to 0} \int_{K} f_\eps^\ast\omega \wedge \varphi,
\]
If $\alpha > \frac{1}{2}$, the limit exists and is independent of the precise approximating sequence $f_\eps$, by \Cref{la:smoothapproxpullback}.
\end{definition}

We now state some consequence of \Cref{la:smoothapproxpullback}.
\begin{lemma}\label{la:diffeodistributionalpullback}
Let $K\subset\R^n$ be either a bounded open set or the closure of a bounded
$C^\infty$ domain, let $\alpha>\frac12$, and let $f\in C^\alpha(K,\R^N)$.

Let $\Phi:\R^N\to\R^N$ be a smooth diffeomorphism and let
\[
   \omega\in C^\infty(\Ep^1\R^N), \quad     \eta:=\Phi^\ast\omega .
\]
Then
\[
        (\Phi\circ f)^\ast\omega=f^\ast\eta
\]
in the distributional sense on $K$.

In particular, if $f^\ast\eta=0$ in the distributional sense on $K$, then $(\Phi\circ f)^\ast\omega=0$ in the distributional sense on $K$.
\end{lemma}
%
%

%
We also record
\begin{lemma}
\label{la:distrpullbackvanishes}
Let $\alpha>1/2$. Suppose that
\[
        F_k\in C^\infty(\B^n;\R^N)
\]
converges to $F$ in $C^\alpha(\B^n;\R^N)$ and that
\begin{equation}\label{eq:distrbpullbackvanis}
        \|F_k^\ast\lambda\|_{L^\infty(\B^n)}\to0.
\end{equation}
Then
\[
        F^\ast\lambda=0 \quad \text{in $\B^n$ in the sense of distributions.}
        \]
\end{lemma}

\subsection{The Constant Rank Theorem for Pfaffian forms}
The following is a version of the constant Rank theorem. For $d=0$ it becomes the Frobenius theorem. See, e.g.,
\cite[Theorem 3.1]{BryantChernGardnerGoldschmidtGriffiths1991}
\begin{theorem}[Constant Rank theorem for Pfaff forms]\label{th:pfaffconstantrank}
Assume $\lambda \in C^\infty(\Ep^1 \R^N)$ is such that in an open set $U$ for some $d \in \{0,1,\ldots,\}$ we have
\[
 \brac{d\lambda}^d \wedge \lambda \neq 0 \quad \text{in $U$}
\]
but
\[
 \brac{d\lambda}^{d+1} \wedge \lambda \equiv 0 \quad \text{in $U$}
\]
Then for any $X_0 \in U$ there exists a diffeomorphism
\[
        p: \R^N \to \R^N
\]
and some $\psi \in C^\infty(\R^N)$
and for possibly smaller open set $V \subset U$, $X_0 \in V$ such that
\begin{equation}\label{eq:pfaff-normal-form}
        \lambda  = \psi \brac{dp^1+p^2\,dp^3+\cdots+p^{2d}\,dp^{2d+1}}\quad \text{in $V$}
\end{equation}
\end{theorem}

\subsection{Proof of Propositions~\ref{pr:Frobeniussituation}, \ref{pr:nonextension}, \ref{pr:nonclose}}
\begin{proof}[Proof of \Cref{pr:Frobeniussituation}] This is a very simple proof, and we give it only for the convenience of the reader.
If $\lambda \equiv 0$ there is nothing to show, so we may assume
$\lambda\neq0$ in $U$. By the $d=0$ case of the constant-rank theorem for Pfaff
forms, \Cref{th:pfaffconstantrank}, we may assume (shrinking to a smaller neighborhood of $X_0$ as needed)
\[
        \lambda=\psi\,dp^1 \quad \text{in $U$}.
\]
We may assume $\psi \neq 0$ in $U$, otherwise we shrink to a smaller neighborhood of $X_0$.
We may also assume that $X_0 = 0$ and $p: \B^{N} \to U \subset \R^N$ is a diffeomorphism that satisfies $p(0)=0$.
Choose $\rho>0$ so small that
\[
        [-2\rho,2\rho]^{N} \Subset p(U),
\]
and set
\[
        V
        :=
        p^{-1}\bigl((-\rho,\rho)^{N}\bigr) \subset U \subset \R^N.
\]
Now assume $f:\partial\mathbb B^n\to V$ is smooth and $f^\ast\lambda=0$. Write
\[
        p=(p^1,p'),
        \qquad
        p'=(p^2,\ldots,p^N).
\]
Since $\psi\circ f\neq0$, we conclude
\[
        d(p^1\circ f)=0 \quad \text{in $U$}.
\]
Since $\partial\mathbb B^n$ is connected for
$n\ge2$, we find $p^1\circ f \equiv c$ for some constant $c\in(-\rho,\rho)$.

Set
\[
        g:=p'\circ f:\partial\mathbb B^n\to [-\rho,\rho]^{N-1} .
\]
The map $g$ is Lipschitz. Extend $g$ to a Lipschitz map
\[
        G:\mathbb B^n\to [-\rho,\rho]^{N-1}
\]
e.g.
\[
        G(x)
        :=
        \begin{cases}
        |x|\,g\!\left(\frac{x}{|x|}\right), & x\neq0,\\
        0, & x=0.
        \end{cases}
\]
Set
\[
        F(x):=p^{-1}\bigl(c,G(x)\bigr),
        \qquad x\in\mathbb B^n .
\]
Then $F$ is Lipschitz and takes values in $V\Subset U$. Moreover, if
$x\in\partial\mathbb B^n$, then
\[
        p(F(x))
        =
        (c,G(x))
        =
        (p^1(f(x)),p'(f(x)))
        =
        p(f(x)).
\]
Since $p$ is a diffeomorphism, $F=f$ on $\partial\mathbb B^n$.

It remains to show $F^\ast\lambda = 0$. We have
\[
        p^1\circ F\equiv c
        \qquad\text{in }\mathbb B^n.
\]
Since $F$ is Lipschitz, it is differentiable a.e. and we have,
\[
        F^\ast\lambda
        =
        (\psi\circ F)\,d(p^1\circ F)
        =
        (\psi\circ F)\,dc
        =
        0 .
\]
Thus $F^\ast\lambda=0$ pointwise a.e. Since $F$ is Lipschitz pointwise is equivalent to distributional and we can conclude.
\end{proof}

For \Cref{pr:nonextension} we use the following combinatorial observation -- essentially a version of Lefshetz Lemma, cf. \cite[Proposition~1.2.30]{huybrechts}, \cite[Proposition~1.1a]{Bryant}, \cite[Lemma 2.9]{HajlaszMirraSchikorra}.

\begin{lemma}\label{la:lefschetz}
Let $d\geq 1$, let $\R^{2d+1}$ have coordinates
$p^1,\ldots,p^{2d+1}$, and set
\[
\lambda=dp^1+\sum_{i=1}^d p^{2i}\,dp^{2i+1}.
\]
Then every smooth $(d+1)$-form $\omega\in C^\infty(\Ep^{d+1} \R^{2d+1})$ can be written as
\[
\omega=\lambda\wedge\widetilde{\omega}_1+d\lambda\wedge\widetilde{\omega}_2
\]
for some smooth forms
$\widetilde{\omega}_1\in C^\infty(\Ep^d\R^{2d+1})$ and
$\widetilde{\omega}_2\in C^\infty(\Ep^{d-1}\R^{2d+1})$. Moreover we have $\supp \tilde{\omega}_1 \cup \supp \tilde{\omega}_2 \subset \supp \omega$.
\end{lemma}
\begin{proof}

Write
\[
\omega=dp^1\wedge\alpha+\beta,
\]
where $\alpha$ and $\beta$ involve only the differentials
$dp^2,\ldots,dp^{2d+1}$. By linear independence we have 
\[
 \supp \alpha \cup \supp \beta \subset \supp \omega.
\]
Since
$dp^1=\lambda-\sum_{i=1}^d p^{2i}\,dp^{2i+1}$,
we get
\[
\omega
=
\lambda\wedge\alpha
+
\left(
\beta-\sum_{i=1}^d p^{2i}\,dp^{2i+1}\wedge\alpha
\right).
\]
The expression in parentheses is a smooth linear combination of monomials of
degree $d+1$ in the variables
\[
dp^2,\ldots,dp^{2d+1}.
\]
Each of these monomials has a coefficient with support in $\supp \alpha \cup \supp \beta \subset \supp \omega$.
It is therefore enough to prove that every such monomial is of the form
$d\lambda\wedge\eta$ for some $(d-1)$-form $\eta$ involving only
$dp^2,\ldots,dp^{2d+1}$.

Fix such a monomial $m$. For each pair
\[
dp^{2i},\ dp^{2i+1},
\qquad 1\leq i\leq d,
\]
there are three possibilities: both occur in $m$, exactly one occurs in $m$, or
neither occurs in $m$. Let
\[
F=\{i:\text{ both }dp^{2i},dp^{2i+1}\text{ occur in }m\},
\]
\[
E=\{i:\text{ neither }dp^{2i},dp^{2i+1}\text{ occurs in }m\},
\]
and let $S$ be the set of indices for which exactly one of
$dp^{2i},dp^{2i+1}$ occurs.

Since $m$ has degree $d+1$,
\[
2|F|+|S|=d+1.
\]
Also, since $F\cup E\cup S=\{1,\ldots,d\}$ and the union is disjoint,
\[
|F|+|E|+|S|=d.
\]
Subtracting gives
\[
|F|=|E|+1.
\]

Let $u$ be the wedge product of the single differentials coming from the indices
in $S$. Up to changing the final sign of $\eta$, we may assume
\[
m
=
u\wedge\bigwedge_{i\in F}
\left(dp^{2i}\wedge dp^{2i+1}\right).
\]
All products over sets of indices are taken in increasing order.

We want to find $\eta_m$ so that
\begin{equation}\label{eq:lefschetz:goal}
 m = d\lambda \wedge \eta_m.
\end{equation}
We take the Ansatz
\[
\eta_m
:=
u\wedge
\sum_{\substack{A\subseteq F\cup E\\ |A|=|E|}}
c_A
\bigwedge_{i\in A}
\left(dp^{2i}\wedge dp^{2i+1}\right),
\]
where the numbers $c_A\in\mathbb Q$ will be chosen now. The degree is correct,
because
\[
\deg\eta_m=|S|+2|E|=|S|+|E|+|F|-1=d-1.
\]

Since
\[
d\lambda=\sum_{i=1}^d dp^{2i}\wedge dp^{2i+1},
\]
and since for every $i\in S$ one of $dp^{2i}$, $dp^{2i+1}$ already occurs in
$u$, the terms with $i\in S$ vanish when wedged with $\eta_m$. Thus
\[
\begin{split}
d\lambda\wedge\eta_m
=&
u\wedge
\left(
\sum_{i\in F\cup E}dp^{2i}\wedge dp^{2i+1}
\right)
\wedge
\sum_{\substack{A\subseteq F\cup E\\ |A|=|E|}}
c_A
\bigwedge_{i\in A}
\left(dp^{2i}\wedge dp^{2i+1}\right)\\
=&\sum_{\substack{B\subseteq F\cup E\\ |B|=|E|+1}} \brac{\sum_{\substack{A\subset B\\ |A|=|E|}} c_A}\,
u\wedge
\bigwedge_{i\in B}
\left(dp^{2i}\wedge dp^{2i+1}\right).
\end{split}
\]
So it is enough for \eqref{eq:lefschetz:goal} to choose the numbers $c_A$ so that
\begin{equation}\label{eq:lefschetz:goalv2}
\sum_{\substack{A\subset B\\ |A|=|E|}} c_A
=
\begin{cases}
1, & B=F,\\
0, & B\neq F.
\end{cases}
\end{equation}

We choose $c_A$ depending only on $|A\cap E|$, $c_A = c_{|A \cap E|}$. Choose first
\[
c_0=\frac{1}{|E|+1},
\]
and then define recursively, for $1\leq r\leq |E|$,
\[
c_r
=
-\frac{r}{|E|+1-r}\,c_{r-1}.
\]
If $B=F$, then since $F \cap E = \emptyset$ and $|F| = |E|+1$,
\[
\sum_{\substack{A\subset B\\ |A|=|E|}}c_{|A\cap E|}
=\sum_{\substack{A\subset F\\ |A|=|E|}}c_{0}
=
(|E|+1)c_0
=
1.
\]
If $B\neq F$, set $r=|B\cap E|$. Then $1\leq r\leq |E|$.  Then $B$ contains $r$ elements of $E$ and $|B|-r=|E|+1-r$ elements of $F$.
Choosing $A \subset B$ is the same as eliminating one element from $B$, so we eleminate $r$ times an elements of $E$ and thus $c_|A \cap E| = c_{r-1}$ and $|E|+1-r$ times we eliminate an element of $F$, i.e. $c_{|A\cap E|} = c_r$. Thus,
\[
\sum_{\substack{A\subset B\\ |A|=|E|}} c_{|A\cap E|}
=
r c_{r-1}+(|E|+1-r)c_r
=
0
\]
by the recursive definition of $c_r$. This gives \eqref{eq:lefschetz:goalv2} and we can conclude.
\end{proof}

\begin{proof}[Proof of \Cref{pr:nonextension}]
This follows essentially from the arguments in \cite{HS2023}, but we reformulate them for the convenience of the reader.

By constant rank theorem, \Cref{th:pfaffconstantrank}, we may assume (making $U$ smaller)
\begin{equation}\label{eq:prnonextensionlambda} \lambda = dp^1 + p^2 dp^3 + \ldots + p^{2d} d p^{2d+1} \end{equation}

Since $d+1\leq 2d+1$, we can find a smooth embedding  $g:\S^{d} \to \R^{2d+1} \times \{0\} \subset \R^N$, $g(\S^d) \subset p(V)$ and $g^\ast\lambda = 0$:

E.g. we can take $g$ the rescaled Legendrian embedding $\gamma$: For $\xi=(\xi_{1},\xi_2,\ldots,\xi_{d+1})\in \S^d\subset\R^{d+1}$, define
\[
        \gamma:\S^d\to\R^{2d+1}
\]
by
\begin{equation}\label{eq:legendrianembedding}
        \gamma^1(\xi)=\frac13\xi_{1}^3,
        \qquad
        \gamma^{2i}(\xi)=\xi_{1}\xi_{i+1},
        \qquad
        \gamma^{2i+1}(\xi)=\xi_{i+1},
        \qquad i=1,\ldots,d.
\end{equation}
Then $\gamma: \S^d \to \R^{d+1}$ is an embedding and $\gamma^\ast \lambda = 0$.

Let $\delta_\rho:\R^{2d+1}\to\R^{2d+1}$ denote the contact dilation
\begin{equation}\label{eq:scaling}
        \delta_\rho(p)^1=\rho^2 p^1,
        \qquad
        \delta_\rho(p)^{2i}=\rho p^{2i},
        \qquad
        \delta_\rho(p)^{2i+1}=\rho p^{2i+1}.
\end{equation}
and set
\[
 g := \delta_\rho \circ (\gamma,0)
\]
where $\rho$ is so small that $g: \S^{d} \to V$.

By \cite[Lemma 2.6.]{HS2023} (observe that $g$ is smooth, so no approximation is needed) there exists
\[
        \omega\in C_c^\infty(\Lambda^{d}\R^{N}), \qquad         d\omega\equiv 0
        \text{ in a neighborhood of }g(\S^{d})
\]
such that ,
\[
        \int_{\S^{d}} g^\ast\omega \neq 0 .
\]
Of course, since $g$ is explicit we could compute $\omega$ explicitly, but it's worth noting that this argument works for any embedding $g:\S^{d+1} \to \R^{2d+1} \times \{0\}$.

Assume now, towards a contradiction, that there is
\[
        F\in C^\alpha(\overline{\B^{d+1}};V), \quad \text{for some }\alpha>\frac{d+1}{d+2}
\]
with $F\Big |_{\partial \B^{d+1}} = g$ and
\[
        F^\ast\lambda=0 \quad \text{in $\B^{d+1}$ in distributional sense}
\]
We may assume that $N=2d+1$, indeed, otherwise compose $F$ with the projection on the first $2d+1$ coordinates, which does not affect $\lambda$ nor $g$.

Observe that $d\omega$ is smooth $d+1$-form, so we may write
\[
 d\omega = \lambda \wedge \tilde{\omega}_1 + d\lambda \wedge \tilde{\omega}_2
\]
by \Cref{la:lefschetz}. 

For any smooth approximation $F_\eps: \B^{d+1} \to V$ of $F$ we may assume that $F_\eps^\ast \tilde{\omega}_i \equiv 0$ close to $\partial \B^{d+1}$, and have 
\[
 \begin{split}
 \int_{\B^{d+1}} F_\eps^\ast d\omega =& \int_{\B^{d+1}} F_\eps^\ast \lambda \wedge F_\eps^\ast \tilde{\omega}_1 + \int_{\B^{d+1}} F_\eps^\ast d\lambda \wedge F_\eps^\ast \tilde{\omega}_2\\
 =&\int_{\B^{d+1}} F_\eps^\ast \lambda \wedge F_\eps^\ast \tilde{\omega}_1 - \int_{\B^{d+1}} F_\eps^\ast \lambda \wedge F_\eps^\ast d\tilde{\omega}_2\\
 \xrightarrow{\eps \to 0}& 0,
 \end{split}
\]
by the assumption that $F^\ast(\lambda) = 0$ in distributional sense, as long as the convergence $F_\eps \to F$ is in $C^\alpha$, $\alpha > \frac{d+1}{d+2}$ using \Cref{la:smoothapproxpullback}.

We thus find
\[
 \int_{\S^d} g^\ast\omega = \lim_{\eps \to 0} \int_{\B^{d+1}} F_\eps^\ast d\omega \equiv 0,
\]
for any smooth approximation of $F$, and thus a contradiction. We can conclude.
\end{proof}

\begin{proof}[Proof of \Cref{pr:nonclose}]
We argue similar to \Cref{pr:nonextension}: The statement is local and we assume w.l.o.g. $\lambda$ has the standard form \eqref{eq:prnonextensionlambda} and $X_0 =0$, using \Cref{th:pfaffconstantrank}.

Let
\[
        g:\S^d\to \R^{2d+1} \times \{0\}_{N-(2d+1)}
\]
be the Legendrian embedding of \eqref{eq:legendrianembedding}, which satisfies $g^\ast\lambda=0$ and let
$G:\R^{d+1}\to\R^{2d+1} \times \{0\}_{N-(2d+1)}$ be the trivial extension
\[
        G^1(y)=\frac13 y_0^3,
        \qquad
        G^{2i}(y)=y_0y_i,
        \qquad
        G^{2i+1}(y)=y_i,
        \qquad i=1,\ldots,d .
\]
Set $F_0 := \delta_\rho \circ G$, where $\delta_\rho$ is the dilation from \eqref{eq:scaling}, we can assume w.l.o.g. that $F_0: \overline{\B^{d+1}} \to V$.

Fix now $\alpha \in (\frac{d+1}{d+2},1)$ and fix $\gamma = \gamma(\alpha,F_0) \in (0,1)$ determined below, and assume $F: \B^{d+1} \to V$ with
\[
 \|F-F_0\|_{C^\alpha(\B^{d+1})} \leq \gamma.
\]
In particular we have
\[
 \|F\|_{C^{\alpha}(\frac{1}{2} \B^{d+1})} \aleq \|F_0\|_{C^\alpha(\B^{d+1})} +1 =: \Lambda
\]
By explicit computation, or \cite[Lemma 2.6]{HS2023}, applied to the smooth embedding $f_0 := F_0 \Big |_{\partial \frac{1}{2}\B^{d+1}} \to \R^N$ there exists
\[
        \omega\in C_c^\infty(\Lambda^d\R^{2d+1}), \quad d\omega\equiv0
        \qquad\text{in a neighborhood } f_0(\frac{1}{2}\S^d),
\]
such that
\[
        \int_{\frac{1}{2}\S^d}f_0^\ast\omega=1.
\]
From \Cref{la:smoothapproxpullback}, since $\alpha > \frac{d}{d+1}$ for any $C^\alpha$-approximation $F_\eps$ of $F$
\[
 \abs{\int_{\frac{1}{2}\S^d}f_0^\ast\omega-\lim_{\eps \to 0} \int_{\frac{1}{2}\S^d}F_\eps^\ast\omega} \aleq_{\Lambda} \|F-F_0\|_{C^\alpha} \leq \gamma.
\]
Thus we have, whenever $\gamma$ is small enough only (depending on $\Lambda$ and thus $F_0$),
\begin{equation}\label{eq:alskdageq12}
 \lim_{\eps \to 0} \int_{\frac{1}{2}\S^d}F_\eps^\ast\omega h^\ast\omega \geq \frac{1}{2}.
\end{equation}
We would like to use the arguments of \Cref{pr:nonextension} that this is incompatible with $F^\ast \lambda = 0$ in $\B^{d+1}$ in distributional sense, because $\alpha > \frac{d+1}{d+2}$. But observe that we do not have the support condition of $d \omega$. Instead, we observe exactly as in \cite[Proof of Theorem 7.4.]{HajlaszMirraSchikorra},
\[
 F^\ast(\lambda) = 0 \text{ in $\B^n$} \quad \Rightarrow \quad \|F_\eps^\ast(\lambda)\|_{L^\infty(\frac{1}{2} \B^n)} \aleq_{\Lambda} \eps^{2\alpha -1},
\]
where $F_\eps$ denotes the usual mollification of $F$. Moreover, 
\[
 \|F_\eps \eta\|_{L^\infty(\overline{\frac{1}{2}\B^n})} \aleq_{\eta} \eps^{(\alpha-1)k} \quad \text{for $\eta \in C^\infty(\Ep^k \R^N)$}
\]

And thus with the help of \Cref{la:lefschetz} we have
\[
\begin{split}
 \int_{\frac{1}{2} \S^d} F_\eps^\ast \omega = &
 \int_{\frac{1}{2} \B^{d+1}} F_\eps^\ast \lambda \wedge F^\ast \tilde{\omega}_1
 +\int_{\frac{1}{2} \B^{d+1}} F_\eps^\ast d\lambda \wedge F^\ast \tilde{\omega}_2\\
 = & \int_{\frac{1}{2} \B^{d+1}} F_\eps^\ast \lambda \wedge F^\ast \tilde{\omega}_1
 -\int_{\frac{1}{2} \B^{d+1}} F_\eps^\ast \lambda \wedge F^\ast d\tilde{\omega}_2
 +\int_{\frac{1}{2} \S^{d}} F_\eps^\ast \lambda \wedge F^\ast \tilde{\omega}_2\\
 \end{split}
\]
which implies 
\[
 \abs{\int_{\frac{1}{2} \S^d} F_\eps^\ast \omega} \aleq \eps^{2\alpha -1+d(\alpha-1)} + \eps^{2\alpha -1+(d-1)(\alpha-1)} \xrightarrow{\eps \to 0} 0
\]
since $\alpha > \frac{d+1}{d+2}$. This is a contradiction to \eqref{eq:alskdageq12} and we can conclude.

\end{proof}
\section{The Nash-construction}\label{s:nash}
Our arguments in the next sections are based on Convex Integration techniques due to Nash, and we profited a lot from the exposition in \cite{Sze12}.

Throughout the remainder of this paper we will work with the following form
\begin{equation}\label{eq:pfaffcr}
        \lambda
        =
        dp^1+
        \sum_{m=1}^d p^{2m}\,dp^{2m+1}
        \qquad\text{on }\R^N .
\end{equation}
We are interested in solving in distributional sense for $F=(F^1,\ldots,F^N): \R^n \to \R^N$
\[
 F^\ast(\lambda) = 0,
\]
i.e.
\[
 \nabla F^1 + \sum_{m=1}^d F^{2m}\,\nabla F^{2m+1} = 0
\]
We observe that the mix of derivatives prevents this from being a first-order differential inclusion of the form $\nabla F \in \mathcal{K}$. Nonetheless, taking the $curl$ it is (on topologically trivial domains) equivalent to a first order differential inclusion -- but we shall make no direct use of that fact and rather work with the equation above.

We begin by observing the following ``geometric lemma'', \Cref{pr:geometriclemma}. It is a version of the elementary observation that we can fix $n+1$ unit vectors $\{\xi_1, \ldots,\xi_{n+1}\} \subset \S^{n-1}$ (with baricenter $0$) such that their convex hull contains a small ball $B_\eps^\ast(0) \subset \R^n$. It is somewhat a variant of \cite[Lemma 3.3]{Sze12}. The finiteness of the sum helps us control how many corrections we need in the iteration. Reducing the size of the sum is thus directly correlated to improving the H\"older exponent of the resulting map.
\begin{proposition}[The geometric lemma]
\label{pr:geometriclemma}
There exist linear functions $\phi_1,\ldots,\phi_{n+1}:\R^n\to\R$, $\phi_j(x) = \langle v_j,x\rangle$ for some $v_j \in \S^{n-1}$, a number $\eps^\ast>0$ and smooth functions
\[
        \gamma_j:B_{\eps^\ast}(0)\subset\R^n\to(0,\infty),
        \qquad j=1,\ldots,n+1,
\]
with the following property:

For every $q=(q_1,\ldots,q_n)\in B_{\eps^\ast}(0)$, one has
\begin{equation}\label{eq:geometriclanice}
        q_1\,dx^1+\cdots+q_n\,dx^n
        =
        \sum_{j=1}^{n+1}\gamma_j(q)^2\,d\phi_j .
\end{equation}
Moreover
\begin{equation}\label{eq:geomlemests}
        0<c\leq \gamma_j(q)\leq C,
        \qquad
        |D\gamma_j(q)|\leq C,
        \qquad
        |D^2\gamma_j(q)|\leq C
\end{equation}
for every $q\in B_{\eps^\ast}(0)$ and every $j=1,\ldots,n+1$, and some dimensional constants $c, C>0$.
\end{proposition}

\begin{proof}
Choose vectors $v_1,\ldots,v_{n+1}\in \S^{n-1}$ to be the vertices of any regular simplex with baricenter at the origin. Thus
\[
        \sum_{j=1}^{n+1} v_j=0,
\]
and the vectors $v_j$ span $\R^n$.
Define $\phi_j(x)=\langle v_j,x\rangle$, for $x\in\R^n$.
Then
\[
        d\phi_j=\sum_{\alpha=1}^n v_j^\alpha\,dx^\alpha,
        \qquad
        |d\phi_j|=|v_j|=1,
\]
and
\begin{equation}\label{eq:sumdphijeq0}
        \sum_{j=1}^{n+1}d\phi_j=0.
\end{equation}

Since the vectors $v_1,\ldots,v_{n+1}$ span $\R^n$, there exist linear
functions
$\ell_j:\R^n\to\R$, $j=1,\ldots,n+1$,
such that
\[
        q=\sum_{j=1}^{n+1}\ell_j(q)v_j
        \qquad\text{for every }q\in\R^n.
\]
In particular, for every $q=(q_1,\ldots,q_n)$,
\[
        \langle q,x\rangle =\sum_{j=1}^{n+1}\ell_j(q) \langle v_j,x \rangle
        \qquad\text{for every }q\in\R^n.
\]
and thus
\begin{equation}\label{eq:linearestdec}
        q_1\,dx^1+\cdots+q_n\,dx^n
        =
        \sum_{j=1}^{n+1}\ell_j(q)\,d\phi_j .
\end{equation}

Choose $\eps^\ast>0$ so that
\[
        |\ell_j(q)|\leq \frac12
        \qquad\text{on }B_{\eps^\ast}(0).
\]
Then
\[
        1+\ell_j(q)>0
        \qquad\text{for }q\in B_{\eps^\ast}(0),
        \quad j=1,\ldots,n+1.
\]
Define $\gamma_j(q):=\sqrt{1+\ell_j(q)}$. Clearly, \eqref{eq:geomlemests} holds.
Moreover, using \eqref{eq:sumdphijeq0} and \eqref{eq:linearestdec}, for any $q \in B_{\eps^\ast}(0)$
\[
\begin{split}
        \sum_{j=1}^{n+1}\gamma_j(q)^2\,d\phi_j
        &=
        \sum_{j=1}^{n+1}\bigl(1+\ell_j(q)\bigr)d\phi_j        \\
        &=
        \sum_{j=1}^{n+1}d\phi_j
        +
        \sum_{j=1}^{n+1}\ell_j(q)d\phi_j                      \\
        &=
        q_1\,dx^1+\cdots+q_n\,dx^n .
\end{split}
\]
This proves \eqref{eq:geometriclanice} and we can conclude.
\end{proof}

The following combinatorial observation is used to improve the H\"older regularity, but if we were to use $s=n+1$ (which is trivial) we'd still get nontrivial H\"older bounds above $1/2$. We recommend to take $s=n+1$ and $\nu(i) = i$ for the first reading.
\begin{lemma}
\label{la:wigglebookkeeper}
Let $b,n \in \N$, and set $s:=\left\lceil\frac{n+1}{b}\right\rceil$.

Set
\begin{equation}\label{eq:bookkeeper:asdopasd2def}
        m_b(j):=1+((j-1)\bmod b),
        \qquad
        \nu_b(j):=1+\left\lfloor\frac{j-1}{b}\right\rfloor .
\end{equation}

Then
\[
        m_b:\{1,\ldots,n+1\}\to\{1,\ldots,b\},
        \qquad
        \nu_b:\{1,\ldots,n+1\}\to\{1,\ldots,s\},
\]
and, whenever $i<j$ and $m(i)=m(j)$, one has
$\nu(i)<\nu(j)$.
\end{lemma}

\begin{proof}
Since $j\leq n+1$, the definition of $s$ gives $\nu(j)\leq s$. If $m(j) = m(i)$ then $j-i = kb$ for some $k \in \Z$. So if $i < j$ then $i = j+kb$ for some $k \in \N$. Thus $\nu(j) > \nu(i)$.
\end{proof}

Convex integration a la Nash-Kuiper \cite{Nash54,Nash56,Kuiper55a,Kuiper55b} relies on updating successively a geometric equation with very simple one-dimensional corrugations. The basis of one corrugation for our situation is the following
\begin{lemma}
\label{la:profile}
There exist smooth $1$-periodic functions
\[
        \eta,\zeta,\chi:\R\to\R
\]
such that
\begin{equation}\label{eq:profiles}
        2\chi+\eta\zeta=0,
        \qquad
        \chi'+\eta\zeta'=1 \quad \text{in $\R$}
\end{equation}
\end{lemma}

\begin{proof}
Take
\[
        \eta(t)=\pi^{-1/2}\cos(2\pi t),
        \qquad
        \zeta(t)=\pi^{-1/2}\sin(2\pi t),
        \qquad
        \chi(t)=-\frac12\eta(t)\zeta(t).
\]
Then $2\chi+\eta\zeta=0$. Also
\[
\begin{split}
        \chi'+\eta\zeta'
        &=
        -\frac12(\eta'\zeta+\eta\zeta')+
        \eta\zeta'
        \\
        &=
        -\frac12\eta'\zeta+
        \frac12\eta\zeta'
        \\
        &=
        \sin^2(2\pi t)+\cos^2(2\pi t)
        =1.
\end{split}
\]
\end{proof}

Now we compute how updates to $F$ change $F^\ast \lambda$,
\begin{lemma}
\label{la:wiggleonce}
Assume $2d+1 \leq N$. Let $F\in C^\infty(\overline{\B^n};\R^N)$, fix $m\in\{1,\ldots,d\}$, and
let $f,u,h\in C^\infty(\B^n)$

Define $H:\B^n\to\R^N$ by
\begin{equation}\label{eq:wiggleonce:H2m}
        H^{2m}=f,
        \qquad
        H^{2m+1}=u,
        \qquad
        H^1=-F^{2m}u+h,
\end{equation}
and
\begin{equation}\label{eq:wiggleonce:nochange}
        H^\ell=0
        \qquad\text{for }\ell\notin\{1,2m,2m+1\}.
\end{equation}
Set $G:=F+H$.

Then
\begin{equation}\label{eq:wiggleonceGastlambda}
        G^\ast\lambda-F^\ast\lambda
        =
        f\,dF^{2m+1}
        -
        u\,dF^{2m}
        +dh+f\,du.
\end{equation}
\end{lemma}

\begin{proof}
Only the coordinates $1,2m,2m+1$ are changed. Hence
\[
\begin{split}
        G^\ast\lambda-F^\ast\lambda
        =&
        dH^1
        +F^{2m}\,dH^{2m+1}
        +H^{2m}\,dF^{2m+1}
        +H^{2m}\,dH^{2m+1}\\
        \overset{\eqref{eq:wiggleonce:H2m}}{=}&
        -F^{2m}\,du-u\,dF^{2m}+dh
        +F^{2m}\,du
        +f\,dF^{2m+1}
        +f\,du\\
        =&
        f\,dF^{2m+1}
        -u\,dF^{2m}
        +dh
        +f\,du\\
\end{split}
\]
We can conclude.
\end{proof}

The following is the fundamental update procedure computation, using the maps from \Cref{la:profile}.
Nash's idea is simple, and it is amazing that it actually works: By \Cref{pr:geometriclemma} we can hope to (locally) write
\[
 F^\ast \lambda = \sum_{i=1}^{n+1} A_i^2 d\phi_i
\]
The following argument provides an update that ``removes'' e.g. $A_1 d\phi_1$, at the expense of a ``small'' error term $E$.

\begin{proposition}[Single update]
\label{pr:oneupdate}
Let $F\in C^\infty(\overline{\B^n};\R^N)$, fix $m\in\{1,\ldots,d\}$, and assume
\begin{equation}\label{eq:singleupdate:Df2mass}
        \|DF^{2m}\|_{L^\infty(\B^n)} + \|DF^{2m+1}\|_{L^\infty(\B^n)}\leq \Gamma,
\end{equation}
and
\begin{equation}\label{eq:singleupdate:F2mest}
        \|F^{2m}\|_{L^\infty(\B^n)}\leq \Upsilon.
\end{equation}
and
\begin{equation}\label{eq:singleupdate:D2F2mest}
        \|D^2F^{2m}\|_{L^\infty(\B^n)}
        +
        \|D^2F^{2m+1}\|_{L^\infty(\B^n)}
        \le
        \Lambda,
\end{equation}

Let $A\in C^\infty(\overline{\B^n})$, $L \geq 1$, $\xi \in \R^n$, $|\xi| = 1$ and set $\phi(x):=\langle \xi,x\rangle_{\R^n}$.

Then there is $H\in C^\infty(\overline{\B^n};\R^N)$ with the following properties:

If $A \in C_c^\infty(\B^n)$ then $H \in C_c^\infty(\B^n)$ and $H$ is supported in
$\operatorname{supp}A$.

If we set $G:=F+H$ we have
\begin{equation}\label{eq:singleupdate:GastlambdamFastlambda}
        G^\ast\lambda-F^\ast\lambda
        =
        A^2\,d\phi+E,
\end{equation}
where for uniform constants $C>0$,

\begin{itemize} \item we have an estimate of the error term $E$

\begin{equation}\label{eq:singleupdatererrorterm}
        \|E\|_{L^\infty(\B^n)}
        \le
        C L^{-1/2}\|A\|_{L^\infty(\B^n)}\,\Gamma.
\end{equation}
and
\begin{equation}\label{eq:oneupate:goaladsldjder}
        \|DE\|_{L^\infty(\B^n)}
        \le
        C\left[
        \left(
        L^{1/2}\|A\|_{L^\infty(\B^n)}
        +
        L^{-1/2}\|DA\|_{L^\infty(\B^n)}
        \right)\Gamma
        +
        L^{-1/2}\|A\|_{L^\infty(\B^n)}\Lambda
        \right].
\end{equation}
\item We have the following estimate for the coordinates which have been changed
\begin{equation}\label{eq:singleupdate:C0est1}
        \|H^{2m}\|_{L^\infty(\B^n)}
        +
        \|H^{2m+1}\|_{L^\infty(\B^n)}
        \le
        C L^{-1/2}\|A\|_{L^\infty(\B^n)},
\end{equation}
and
\begin{equation}\label{eq:singleupdate:C1est1}
        \|DH^{2m}\|_{L^\infty(\B^n)}
        +
        \|DH^{2m+1}\|_{L^\infty(\B^n)}
        \le
        C\left(
        L^{1/2}\|A\|_{L^\infty(\B^n)}
        +
        L^{-1/2}\|DA\|_{L^\infty(\B^n)}
        \right).
\end{equation}
\begin{equation}\label{eq:singleupdate:C2asdasdest1}
\begin{split}
        \|D^2 H^{2m}\|_{L^\infty(\B^n)}
        +\|D^2 H^{2m+1}\|_{L^\infty(\B^n)}
        \leq C\Big(
        &L^{3/2}\|A\|_{L^\infty(\B^n)}
        +L^{1/2}\|DA\|_{L^\infty(\B^n)}
        \\
        &+L^{-1/2}\|D^2A\|_{L^\infty(\B^n)}
        \Big).
\end{split}
\end{equation}
\item
and we have the overall estimate
\begin{equation}\label{eq:singleupdate:HLinft21asd}
        \|H\|_{L^\infty(\B^n)}
        \le
        C\, (1+\Upsilon)
        \left(
        L^{-1/2}\|A\|_{L^\infty(\B^n)}
        +
        L^{-1}\|A\|_{L^\infty(\B^n)}^2
        \right),
\end{equation}
and
\begin{equation}\label{eq:singleupdate:DHLinftyasldkj}
\begin{split}
        \|DH\|_{L^\infty(\B^n)}
        \leq C\, (1+\Upsilon) \Big(
        &L^{1/2}\|A\|_{L^\infty(\B^n)}
        +L^{-1/2}\|DA\|_{L^\infty(\B^n)}
        \\
        &+\Gamma L^{-1/2}\|A\|_{L^\infty(\B^n)}
        +L^{-1}\|A\|_{L^\infty(\B^n)}\|DA\|_{L^\infty(\B^n)}
        +\|A\|_{L^\infty(\B^n)}^2
        \Big),
\end{split}
\end{equation}
\item Moreover only three coordinate have changed, i.e. we have \eqref{eq:wiggleonce:nochange}.
\end{itemize}
\end{proposition}

\begin{proof}
Let $\eta,\zeta,\chi$ be the corrugation profiles from \Cref{la:profile} and set for $x \in \B^n$
\begin{equation}\label{eq:singleupdate:fuhdef}
        f(x)=L^{-1/2}\, A(x)\, \eta(L\phi(x)),
        \qquad
        u(x)=L^{-1/2}\, A(x)\, \zeta(L\phi(x)),
        \qquad
        h(x)=L^{-1}\, A^2(x)\, \chi(L\phi(x)).
\end{equation}
Apply \Cref{la:wiggleonce}, and obtain $H$ and $G:=F+H$.
If $A\in C_c^\infty(\B^n)$, then
$H\in C_c^\infty(\B^n;\R^N)$, $\operatorname{supp}H\subset\operatorname{supp}A$. If $A\in C^\infty(\overline{\B^n})$ we still have $H\in C^\infty(\overline{\B^n};\R^N)$

In either case, by
\eqref{eq:wiggleonceGastlambda},
\begin{equation}\label{eq:singleupdate:GastmFastcomp3easd}
        G^\ast\lambda-F^\ast\lambda
        =
        f\,dF^{2m+1}-u\,dF^{2m}+dh+f\,du.
\end{equation}
We compute
\[
\begin{split}
        dh
        &=
        L^{-1}\, 2A\,\chi(L\phi)\,dA
        +A^2\, \chi'(L\phi)\,d\phi,
        \\
        du&=L^{-1/2}\, dA\, \zeta(L\phi) + L^{1/2}\, A\, \zeta'(L\phi)\, d\phi,
        \\
        f\,du
        &=
        L^{-1}\, A\, \eta(L\phi)\, \zeta(L\phi)\,dA
        +A^2\, \eta(L\phi)\, \zeta'(L\phi)\,d\phi.
\end{split}
\]
Using \eqref{eq:profiles},
\[
\begin{split}
        dh+f\,du
        &=
        L^{-1}\, A\, \underbrace{(2\chi+\eta\zeta)(L\phi)}_{\equiv 0}\,dA
        +A^2\, \underbrace{(\chi'+\eta\zeta')(L\phi)}_{\equiv 1}\,d\phi
        \\
        &=A^2\,d\phi.
\end{split}
\]
Thus \eqref{eq:singleupdate:GastmFastcomp3easd} becomes
\[
        G^\ast\lambda-F^\ast\lambda
        =
        A^2\,d\phi
        +f\,dF^{2m+1}-u\,dF^{2m}.
\]
This proves \eqref{eq:singleupdate:GastlambdamFastlambda} with
\begin{equation}\label{eq:Edefonewiggle}
        E:=f\,dF^{2m+1}-u\,dF^{2m}.
\end{equation}

Observe \begin{equation}\label{eq:klxcvixjcvffuest}
        \|f\|_{L^\infty}+\|u\|_{L^\infty}\aleq_{\chi,\eta,\zeta}  L^{-1/2}\|A\|_{L^\infty(\B^n)},
\end{equation}
and \[
        \|h\|_{L^\infty}\aleq_{\chi,\eta,\zeta}  L^{-1}\|A\|_{L^\infty(\B^n)}^2.
\]
Then \eqref{eq:singleupdatererrorterm} follows from
\eqref{eq:singleupdate:Df2mass}.

Moreover, by \eqref{eq:wiggleonce:H2m}
\[
 \|H\|_{L^\infty} \aleq_{\chi,\eta,\zeta}  \brac{L^{-1/2}\|A\|_{L^\infty(\B^n)}+L^{-1}\|A\|_{L^\infty(\B^n)}^2} \brac{1+ \|F^{2m}\|_{L^\infty}}
\]
so \eqref{eq:singleupdate:HLinft21asd} follows.

Recall \eqref{eq:singleupdate:fuhdef}. Observe
 \begin{equation}\label{eq:klxcvixjcvfDfuest}
\begin{split}
        \|Df\|_{L^\infty}+\|Du\|_{L^\infty}
        &\aleq
        L^{-1/2}\|DA\|_{L^\infty}+L^{1/2}\|A\|_{L^\infty}
        \\
        \|Dh\|_{L^\infty}
        &\aleq
        L^{-1}\|A\|_{L^\infty}\,\|DA\|_{L^\infty}+\|A\|_{L^\infty}^2.
\end{split}
\end{equation}
so again by \eqref{eq:wiggleonce:H2m}
\[
\begin{split}
\|DH\|_{L^\infty} \aleq& \|Df\|_{L^\infty} + \|Du\|_{L^\infty} + \|DF^{2m}\|_{L^\infty} \|u\|_{L^\infty} +  \|F^{2m}\|_{L^\infty} \|Du\|_{L^\infty} + \|Dh\|_{L^\infty}\\
\aleq&
(1+\Upsilon) \brac{L^{-1/2}\|DA\|_{L^\infty}+
L^{1/2}\|A\|_{L^\infty} } +
\Gamma  L^{-1/2}\|A\|_{L^\infty(\B^n)} +   L^{-1}\|A\|_{L^\infty}\,\|DA\|_{L^\infty}+\|A\|_{L^\infty}^2\\
\end{split}
\]
which implies \eqref{eq:singleupdate:DHLinftyasldkj}.

Lastly,
\[
        \|D^2f\|_{L^\infty}+\|D^2u\|_{L^\infty}
        \aleq
        L^{-1/2}\|D^2A\|_{L^\infty}+L^{1/2}\|DA\|_{L^\infty}+L^{3/2}\|A\|_{L^\infty}
        ,
\]
which proves \eqref{eq:singleupdate:C2asdasdest1}  by representation \eqref{eq:wiggleonce:H2m}.

Similarly we prove
\eqref{eq:oneupate:goaladsldjder},
\eqref{eq:singleupdate:C0est1},
\eqref{eq:singleupdate:C1est1}, and
\eqref{eq:wiggleonce:nochange}.

By \eqref{eq:Edefonewiggle} and then \eqref{eq:klxcvixjcvffuest},\eqref{eq:klxcvixjcvfDfuest}
\[
\begin{split}
        \|DE\|_{L^\infty}
        \aleq&
        \|Df\|_{L^\infty} \|DF^{2m+1}\|_{L^\infty}
        +\|f\|_{L^\infty}\,\|D^2F^{2m+1}\|_{L^\infty}
        \\
        &+
        \|Du\|_{L^\infty} \|DF^{2m}\|_{L^\infty}
        +
        \|u\|_{L^\infty}\,\|D^2F^{2m}\|_{L^\infty}\\
        \aleq&
         \brac{L^{-1/2}\|DA\|_{L^\infty}+L^{1/2}\|A\|_{L^\infty}} \|DF^{2m+1}\|_{L^\infty}
        +L^{-\frac{1}{2}} \|A\|_{L^\infty}\,\|D^2F^{2m+1}\|_{L^\infty}
        \\
        \aleq&\brac{L^{-1/2}\|DA\|_{L^\infty}+L^{1/2}\|A\|_{L^\infty}} \Gamma
        +L^{-\frac{1}{2}} \|A\|_{L^\infty}\,\Lambda
        \\
\end{split}
\]
This is \eqref{eq:oneupate:goaladsldjder}.

Next, by \eqref{eq:wiggleonce:H2m},
\[
        H^{2m}=f,
        \qquad
        H^{2m+1}=u.
\]
Hence
\[
\begin{split}
        \|H^{2m}\|_{L^\infty(\B^n)}
        +
        \|H^{2m+1}\|_{L^\infty(\B^n)}
        &=
        \|f\|_{L^\infty(\B^n)}
        +
        \|u\|_{L^\infty(\B^n)}
        \overset{\eqref{eq:klxcvixjcvffuest}}{\aleq}
        L^{-1/2}\|A\|_{L^\infty(\B^n)}.
\end{split}
\]
This proves \eqref{eq:singleupdate:C0est1}. Similarly,
\[
\begin{split}
        \|DH^{2m}\|_{L^\infty(\B^n)}
        +
        \|DH^{2m+1}\|_{L^\infty(\B^n)}
        &=
        \|Df\|_{L^\infty(\B^n)}
        +
        \|Du\|_{L^\infty(\B^n)}
        \\
        &\overset{\eqref{eq:klxcvixjcvfDfuest}}{\aleq}
        L^{1/2}\|A\|_{L^\infty(\B^n)}
        +
        L^{-1/2}\|DA\|_{L^\infty(\B^n)}.
\end{split}
\]
This proves \eqref{eq:singleupdate:C1est1}.

Finally, \eqref{eq:wiggleonce:nochange} is immediate from the
definition of $H$.
\end{proof}

 If we iterate the argument of \Cref{pr:oneupdate}, we can try to make $F^\ast \lambda$ arbitrarily small. Of course, we have to be careful with interaction terms due to the nonlinearity of $F^\ast \lambda$ (i.e. two updates $(F+H_1 + H_2)^\ast \lambda \neq (F+H_1)^\ast \lambda + (H^2)^\ast \lambda$, but there are $H_1,H_2$ interaction terms, and this is precisely responsible for the loss in regularity).

\begin{proposition}
\label{pr:severalupdates1}
Recall \eqref{eq:pfaffcr} and that $N\ge 2d+1$, $n \geq 2$. Fix $m \in \{1,\ldots,d\}$. There exists $\eps_0 > 0$ depending on $d$ and $N$ and $n$, otherwise uniform, constant so that the following holds for $s:=\left\lceil\frac{n+1}{d}\right\rceil$:

Let
\[
        F\in C^\infty(\overline{\B^n};\R^N),
        \qquad
        M\ge1,
        \qquad
        \Gamma\ge1,
        \qquad
        \Upsilon\ge0,
        \qquad
        r>0.
\]
and assume for some compact set $K$
\[
        \operatorname{supp}F^\ast\lambda\subset K \Subset \B^n.
\]
Set
\begin{equation}\label{eq:severalupdates1:elldef}
        \ell:=\frac{r}{M\Gamma^2}.
\end{equation}
Assume the following scale compatibility conditions:
\begin{equation}\label{eq:severalupdates1:rleqMGammadef}
        r\leq M\Gamma,
        \qquad
        r\leq M^2\Gamma^2,
        \qquad
        4\ell<\dist(K,\partial \B^n).
\end{equation}
Assume moreover that
\begin{equation}\label{eq:severalupdates1:Fasttlambdas0fdi2}
        \|F^\ast\lambda\|_{L^\infty(\B^n)}
        \le
        \eps_0 r,
\end{equation}
and that
\begin{equation}\label{eq:severalupdates1:DF2mLinftyest}
        \max_{1\leq m\leq d}
        \|D(F^{2m},F^{2m+1})\|_{L^\infty(\B^n)}
        \le
        \Gamma,
\end{equation}
\begin{equation}\label{eq:severalupdates1:D2F2mLinftyestD2ver}
        \max_{1\leq m\leq d}
        \|D^2(F^{2m},F^{2m+1})\|_{L^\infty(\B^n)}
        \le
        \frac{\Gamma^3}{r},
\end{equation}
\begin{equation}\label{eq:severalupdates1:Dfastlambda}
        \|D(F^\ast\lambda)\|_{L^\infty(\B^n)}
        \le
        \Gamma^2,
\end{equation}
and
\begin{equation}\label{eq:severalupdates1:F2mlinftyupsilonasd}
        \max_{1\leq m\leq d}
        \|F^{2m}\|_{L^\infty(\B^n)}
        \le
        \Upsilon.
\end{equation}
Then there exists $G\in C^\infty(\overline{\B^n};\R^N)$
such that
\begin{equation}\label{eq:severalupdates1:outputsupp}
        \operatorname{supp}(G-F)
        \subset
        \{x\in \B^n:\dist(x,K)<2\ell\}
        \Subset \B^n,
\end{equation}
\begin{equation}\label{eq:severalupdates1:outputGastlambda}
        \|G^\ast\lambda\|_{L^\infty(\B^n)}
        \le
        C\frac{r}{M},
\end{equation}
\begin{equation}\label{eq:severalupdates1:outputGmF}
        \|G-F\|_{L^\infty(\B^n)},
        \le
        C(1+\Upsilon)\frac{r}{M\Gamma},
\end{equation}
\begin{equation}\label{eq:severalupdates1:outputDGmF}
        \|D(G-F)\|_{L^\infty(\B^n)},
        \le
        C(1+\Upsilon)M^s\Gamma,
\end{equation}
\begin{equation}\label{eq:severalupdates1:outputDG2nNsGanna}
        \max_{1\leq m\leq d}
        \|D(G^{2m},G^{2m+1})\|_{L^\infty(\B^n)}
        \le
        C M^s\Gamma,
\end{equation}
\begin{equation}\label{eq:severalupdates1:outputD2G2nNsGanna}
        \max_{1\leq m\leq d}
        \|D^2(G^{2m},G^{2m+1})\|_{L^\infty(\B^n)}
        \le
        C\frac{M^{3s}\Gamma^3}{r},
\end{equation}
and
\begin{equation}\label{eq:severalupdates1:outputDGastlambdasad33e}
        \|D(G^\ast\lambda)\|_{L^\infty(\B^n)}
        \le
        C M^{2s}\Gamma^2.
\end{equation}
The following coordinates remain unchanged
\begin{equation}\label{eq:severalupdates1:outputnochange}
        G^j=F^j
        \qquad\text{for }j=2d+2,\ldots,N.
\end{equation}
Here the constants $C$ are independent of
$r,M,\Gamma,\Upsilon,F,\Upsilon$.
\end{proposition}

\begin{proof}
By \eqref{eq:severalupdates1:rleqMGammadef}, the $\ell$-neighborhood of
$\operatorname{supp}F^\ast \lambda$ is compactly contained in $\B^n$. Choose
\[
        \theta\in C_c^\infty(\B^n)
\]
such that for $K := \operatorname{supp} F^\ast \lambda$,
\[
        \theta=1
        \quad\text{on }\{x:\dist(x,K)\leq \ell\},
        \qquad
        \operatorname{supp}\theta
        \subset
        \{x:\dist(x,K)<2\ell\},
\]and
\begin{equation}\label{eq:severalupdates1:cutoffargbdd}
        \|D^a\theta\|_{L^\infty(\B^n)}
        \le
        C\ell^{-a},
        \qquad
        a=1,2.
\end{equation}

Let $\rho$ be a standard mollifier supported in the unit ball. Extend the
coefficient vector of $\omega=F^\ast\lambda$ by zero outside $\B^n$, and define (componentwise)
\[
        \omega^\ell:=\rho_\ell\ast F^\ast (\lambda).
\]
Then $\operatorname{supp}\omega^\ell
        \subset
        \{x:\dist(x,K)\leq \ell\}$, in particular $\theta\equiv 1$ on $\operatorname{supp}\omega^\ell$.

The standard mollification estimates and
\eqref{eq:severalupdates1:Dfastlambda} give
\begin{equation}\label{eq:severalupdates1:mollyerrorasd}
        \|F^\ast \lambda-\omega^\ell\|_{L^\infty(\B^n)}
        \aleq
        \ell\|D(F^\ast \lambda)\|_{L^\infty(\B^n)}
        \aleq
        \frac{r}{M}.
\end{equation}

Moreover, using \eqref{eq:severalupdates1:Fasttlambdas0fdi2} (and $\eps_0 \leq 1$), we have
\begin{equation}\label{eq:severalupdates1:mollyeestD2as}
        \|D^a\omega^\ell\|_{L^\infty(\B^n)}
        \aleq
        r\ell^{-a},
        \qquad
        a=0,1,2.
\end{equation}

From \eqref{eq:severalupdates1:Fasttlambdas0fdi2} we also have
\[
 \left \|\frac{\omega^\ell}{r} \right \|_{L^\infty(\B^n)} \leq \eps_0,
\]
so for $\eps_0$ uniform, but sufficiently small we can assume that all coefficients of $\frac{\omega^\ell}{r}$ belong to $B_{\eps^\ast}(0)$ defined in \Cref{pr:geometriclemma}, and we can write
\[
 -\omega^\ell = \sum_{j=1}^{n+1} r\brac{\gamma_j\brac{-\omega^\ell/r}}^2 d\phi_j.
\]

For $j=1,\ldots,n+1$, set
\begin{equation}\label{eq:severalupdates1:Ajdef}
        A_j(x)
        =
        r^{1/2}\theta(x)
        \gamma_j\left(-\frac{\omega^\ell(x)}{r}\right)
\end{equation}
Since $\theta\equiv1$ on $\operatorname{supp}\omega^\ell$, \Cref{pr:geometriclemma}, gives
\begin{equation}\label{eq:stage-principal-cancellation}
        -\omega^\ell  = \sum_{j=1}^{n+1} A_j^2\,d\phi_j \quad \text{in $\B^n$}
\end{equation}

From \eqref{eq:geomlemests},
\eqref{eq:severalupdates1:cutoffargbdd}, and
\eqref{eq:severalupdates1:mollyeestD2as}, we obtain
\begin{equation}\label{eq:severalupdates1:AestLinftycrsqert}
        \|A_j\|_{L^\infty(\B^n)}
        \le
        C r^{1/2},
\end{equation}

\begin{equation}\label{eq:severalupdates1:DAestLinfty}
        \|DA_j\|_{L^\infty(\B^n)}
        \aleq
        r^{1/2}\ell^{-1}
        +
        r^{-1/2}\|D\omega^\ell\|_{L^\infty(\B^n)}
        \aleq r^{1/2}\ell^{-1} \overset{\eqref{eq:severalupdates1:elldef}}{\aleq}
        \frac{M\Gamma^2}{r^{1/2}},
\end{equation}
and
\begin{equation}\label{eq:severalupdates1:D2Aasdx}
\begin{split}
        \|D^2A_j\|_{L^\infty(\B^n)}
        &\aleq
        r^{1/2}\ell^{-2}
        +
         r^{-1/2}\ell^{-1}
        \|D\omega^\ell\|_{L^\infty(\B^n)}
        \\
        &\quad
        +
         r^{-3/2}
        \|D\omega^\ell\|_{L^\infty(\B^n)}^2
        +
         r^{-1/2}
        \|D^2\omega^\ell\|_{L^\infty(\B^n)}
        \\
        &\aleq
r^{1/2}\ell^{-2}
        \\
        &\overset{\eqref{eq:severalupdates1:elldef}}{\aleq}
        \frac{M^2\Gamma^4}{r^{3/2}}.
\end{split}
\end{equation}

For $j=1,\ldots,n+1$, recall from \Cref{la:wigglebookkeeper},
\[
        m:\{1,\ldots,n+1\}\to\{1,\ldots,d\},
        \qquad
        \nu:\{1,\ldots,n+1\}\to\{1,\ldots,s\},
\]
and, whenever $i<j$ and $m(i)=m(j)$, one has
$\nu(i)<\nu(j)$.

Define
\begin{equation}\label{eq:severalupdates1:Lfrequencydef}
        L_j
        :=
        \frac{M^{2\nu(j)}\Gamma^2}{r},
        \qquad j=1,\ldots,n+1.
\end{equation}
By \eqref{eq:severalupdates1:rleqMGammadef}, $L_j\ge1$.

Starting with $F_0:=F$, apply
\Cref{pr:oneupdate} successively for
$j=1,\ldots,n+1$, using the data
\[
        F_{j-1},
        \qquad
        m(j),
        \qquad
        A_j,
        \qquad
        \phi_j,
        \qquad
        L_j .
\]

To be precise:
Choose a constant $C_0\ge1$, depending only on $n,d,N$ and on the fixed
profiles in \Cref{pr:oneupdate}, large enough
for the induction below.

Set
\[
        F_0:=F.
\]
For $j=0,\ldots,n+1$ and $m\in\{1,\ldots,d\}$ define
\begin{equation}\label{eq:severalupdates1:bookkkeeperkappadef}
        \kappa_j(m)
        :=
        \max\bigl(\{0\}\cup\{\nu(i):1\leq i\leq j,\ m(i)=m\}\bigr).
\end{equation}
Thus $\kappa_0(m)=0$ for every $m$. Moreover, by the properties of $m(\cdot)$ and $\nu(\cdot)$ in
\Cref{la:wigglebookkeeper},
\begin{equation}\label{eq:severalupdates1:kappavsnujest}
        \kappa_{j-1}(m(j))\leq \nu(j)-1 \quad \forall j \geq 1
\end{equation}

We also observe, importantly,
\begin{equation}\label{eq:stagekappaleqs}
\kappa_j(m) \leq s \quad \forall j, m
\end{equation}


We construct $F_j$ inductively so that, for every $j=0,\ldots,n+1$

\begin{itemize} \item For every
$m\in\{1,\ldots,d\}$,
\begin{equation}\label{eq:severalupdates1:inductionDFj2mLinfty}
        \|DF_j^{2m}\|_{L^\infty(\B^n)}
        +
        \|DF_j^{2m+1}\|_{L^\infty(\B^n)}
        \le
        C_j M^{\kappa_j(m)}\Gamma,
\end{equation}
\begin{equation}\label{eq:severalupdates1:inductionD2Fj2mLinfty}
        \|D^2F_j^{2m}\|_{L^\infty(\B^n)}
        +
        \|D^2F_j^{2m+1}\|_{L^\infty(\B^n)}
        \le
        C_j\frac{M^{3\kappa_j(m)}\Gamma^3}{r},
\end{equation}
and
\begin{equation}\label{eq:severalupdates1:inductionFj2mLinftyCj}
        \|F_j^{2m}\|_{L^\infty(\B^n)}
        \le
        C_j(1+\Upsilon).
\end{equation}

Here $C_j$ is a constant that will change with $j$ but be a product of uniform constants.

\item we have
\begin{equation}\label{eq:FjmFsupport}
        \operatorname{supp}(F_j-F)
        \subset
        \operatorname{supp}\theta
        \subset
        \{x\in \B^n:\dist(x,K)<2\ell\}.
\end{equation}
\item
\begin{equation}\label{eq:goal:stagepassiveunchanged}
        F_j^k=F^k
        \qquad\text{for }k=2d+2,\ldots,N.
\end{equation}
\item If $j \geq 1$,
\begin{equation}\label{eq:stageREMstep}
        F_j^\ast\lambda-F_{j-1}^\ast\lambda
        =
        A_j^2\,d\phi_j+E_j.
\end{equation}
where
\begin{equation}\label{eq:goal:Ejest1}
        \|E_j\|_{L^\infty(\B^n)}
        \aleq
        \frac{r}{M}.
\end{equation}
\begin{equation}\label{eq:goal:Ejest2}
        \|DE_j\|_{L^\infty(\B^n)}
        \aleq
        M^{2s}\Gamma^2.
\end{equation}
\item We have \begin{equation}
 \label{eq:Hjest}
 \|H_j\|_{L^\infty(\B^n)} \aleq
        (1+\Upsilon)\frac{r}{M^{\nu(j)}\Gamma}
\end{equation}
\begin{equation}
 \label{eq:Hjestv2}
        \|DH_j\|_{L^\infty(\B^n)}
        \aleq
        (1+\Upsilon)M^{\nu(j)}\Gamma.
\end{equation}

\end{itemize}
\underline{For $j=0$}, these estimates follow from
\eqref{eq:severalupdates1:DF2mLinftyest},
\eqref{eq:severalupdates1:D2F2mLinftyestD2ver}, and
\eqref{eq:severalupdates1:F2mlinftyupsilonasd}, for a uniform $C_0>0$ if necessary. \eqref{eq:FjmFsupport}, \eqref{eq:goal:stagepassiveunchanged} are trivial.

\eqref{eq:goal:Ejest1}, \eqref{eq:goal:Ejest2}, \eqref{eq:stageREMstep} are empty.

\underline{Induction $(j-1)\to j$}:
Let $j\ge1$, and assume that $F_0,\ldots,F_{j-1}$ have been constructed and
satisfy
\eqref{eq:severalupdates1:inductionDFj2mLinfty},
\eqref{eq:severalupdates1:inductionD2Fj2mLinfty}, and
\eqref{eq:severalupdates1:inductionFj2mLinftyCj}.

Set
\[
        \Gamma_j:=C_{j-1}M^{\nu(j)-1}\Gamma,
        \qquad
        \Upsilon_j:=C_{j-1}(1+\Upsilon),
        \qquad
        \Lambda_j
        :=
        C_{j-1}\frac{M^{3(\nu(j)-1)}\Gamma^3}{r}.
\]
By \eqref{eq:severalupdates1:kappavsnujest} and the induction hypothesis,
\[        \|DF_{j-1}^{2m(j)}\|_{L^\infty(\B^n)}
        +
        \|DF_{j-1}^{2m(j)+1}\|_{L^\infty(\B^n)}
        \le
        C_{j-1}M^{\kappa_{j-1}(m(j))}\Gamma
        \le
        \Gamma_j,
\]
\[        \|F_{j-1}^{2m(j)}\|_{L^\infty(\B^n)}
        \le
        C_{j-1}(1+\Upsilon)
        =
        \Upsilon_j,
\]
and\[
        \|D^2F_{j-1}^{2m(j)}\|_{L^\infty(\B^n)}
        +
        \|D^2F_{j-1}^{2m(j)+1}\|_{L^\infty(\B^n)}
        \le
        C_{j-1}\frac{M^{3\kappa_{j-1}(m(j))}\Gamma^3}{r}
        =
        \Lambda_j.
\]
Therefore we can apply \Cref{pr:oneupdate}
\[
        F_{j-1},
        \qquad
        m(j),
        \qquad
        A_j,
        \qquad
        L_j,
        \qquad
        \phi_j\equiv \langle \xi_j,x\rangle
\]
with input constants $\Gamma_j,\Upsilon_j,\Lambda_j$.

Thus we obtain
\[
        H_j\in C_c^\infty(\B^n;\R^N),
        \qquad
        \operatorname{supp}H_j\subset \operatorname{supp}A_j,
\]
and we set
\[        F_j:=F_{j-1}+H_j.
\]
By \eqref{eq:wiggleonce:nochange}, all components of $H_j$ except $1,2m(j),2m(j)+1$ are zero. In particular we have \eqref{eq:goal:stagepassiveunchanged}. Also \eqref{eq:FjmFsupport} is true since $\operatorname{supp} A_j \subset \operatorname{supp} \theta$

By \eqref{eq:singleupdate:GastlambdamFastlambda} we have \eqref{eq:stageREMstep}

Moreover, \eqref{eq:singleupdatererrorterm} and
\eqref{eq:oneupate:goaladsldjder}, applied with
$\Gamma_j,\Lambda_j$, give
\[
        \begin{split}
        \|E_j\|_{L^\infty(\B^n)}
        \aleq&
        L_j^{-1/2}\|A_j\|_{L^\infty(\B^n)}\Gamma_j\\
        \aleq& \frac{M^{-\nu(j)}\Gamma^{-1}}{r^{-1/2}} r^{1/2} M^{\nu(j)-1}\Gamma\\
        =&\frac{r}{M}\\
        \end{split}
\]
This is \eqref{eq:goal:Ejest1}

Also (observe that $\nu(j) \leq s$)
\[
\begin{split}
        \|DE_j\|_{L^\infty(\B^n)}
        \aleq&
        \left(
        L_j^{1/2}\|A_j\|_{L^\infty(\B^n)}
        +
        L_j^{-1/2}\|DA_j\|_{L^\infty(\B^n)}
        \right)\Gamma_j
        +
        L_j^{-1/2}\|A_j\|_{L^\infty(\B^n)}\Lambda_j
        \\
        \aleq&
        \left(
        \brac{\frac{M^{2\nu(j)}\Gamma^2}{r}}^{1/2} r^{1/2}
        +
        \brac{\frac{M^{2\nu(j)}\Gamma^2}{r}}^{-1/2} \frac{M\Gamma^2}{r^{1/2}}
        \right)M^{\nu(j)-1}\Gamma
        +
        \brac{\frac{M^{2\nu(j)}\Gamma^2}{r}}^{-1/2} r^{1/2} \frac{M^{3(\nu(j)-1)}\Gamma^3}{r}
        \\
        \aleq&
        M^{2\nu(j)}\Gamma^2\\
        \aleq&M^{2s} \Gamma^2.
\end{split}
\]
This is \eqref{eq:goal:Ejest2}.

We next record the changed coordinate estimates from
\Cref{pr:oneupdate}. From
\eqref{eq:singleupdate:C0est1},
\eqref{eq:singleupdate:C1est1}, and
\eqref{eq:singleupdate:C2asdasdest1}, and using \eqref{eq:severalupdates1:Lfrequencydef}, and then \eqref{eq:severalupdates1:AestLinftycrsqert},
\eqref{eq:severalupdates1:DAestLinfty}, and
\eqref{eq:severalupdates1:D2Aasdx} we have
\begin{equation}\label{eq:changeHjC0}
        \|H_j^{2m(j)}\|_{L^\infty(\B^n)}
        +
        \|H_j^{2m(j)+1}\|_{L^\infty(\B^n)}
        \aleq
        L_j^{-1/2}\|A_j\|_{L^\infty(\B^n)} \overset{\eqref{eq:severalupdates1:AestLinftycrsqert}}{\aleq} \frac{r}{M^{\nu(j)}\Gamma}
\end{equation}
\begin{equation}\label{eq:DHj2mjest23f}
\begin{split}
        \|DH_j^{2m(j)}\|_{L^\infty(\B^n)}
        +
        \|DH_j^{2m(j)+1}\|_{L^\infty(\B^n)}
        \aleq&
        L_j^{1/2}\|A_j\|_{L^\infty(\B^n)}
        +
        L_j^{-1/2}\|DA_j\|_{L^\infty(\B^n)}
        \\
        \overset{\eqref{eq:severalupdates1:DAestLinfty}}{\aleq}&
        \brac{\frac{M^{2\nu(j)}\Gamma^2}{r}}^{1/2} r^{1/2}
        +
        \brac{\frac{M^{2\nu(j)}\Gamma^2}{r}}^{-1/2} \frac{M\Gamma^2}{r^{1/2}}
        \\
        =&
        M^{\nu(j)}\Gamma
        +
        M^{1-\nu(j)} \Gamma
        \\
        \overset{M \geq 1}{\aleq}&        M^{\nu(j)} \Gamma.
\end{split}
\end{equation}
and
\begin{equation}\label{eq:D2Hjest}
\begin{split}
        &\|D^2H_j^{2m(j)}\|_{L^\infty(\B^n)}
        +
        \|D^2H_j^{2m(j)+1}\|_{L^\infty(\B^n)}\\
        \aleq &
        L_j \brac{L_j^{1/2}\|A_j\|_{L^\infty(\B^n)}
        +
        L_j^{-1/2}\|DA_j\|_{L^\infty(\B^n)}}
        +
        L_j^{-1/2}\|D^2A_j\|_{L^\infty(\B^n)}
        \\
        \aleq&L_j\, M^{\nu(j)} \Gamma
        +
        L_j^{-1/2} \frac{M^2\Gamma^4}{r^{3/2}}
        \\
\aleq&\frac{M^{3\nu(j)}\Gamma^3}{r}\,
        +
         \frac{M^{2-\nu(j)}\, \Gamma^3}{r}
        \\
\aleq&\frac{M^{3\nu(j)}\Gamma^3}{r}\,
        \end{split}
\end{equation}

\eqref{eq:singleupdate:HLinft21asd}, applied with $\Upsilon_j$, and by
\eqref{eq:severalupdates1:AestLinftycrsqert} and \eqref{eq:severalupdates1:Lfrequencydef},
\[
\begin{split}
        \|H_j\|_{L^\infty(\B^n)}
        &\aleq
        (1+\Upsilon_j)
        \left(
        L_j^{-1/2}\|A_j\|_{L^\infty(\B^n)}
        +
        L_j^{-1}\|A_j\|_{L^\infty(\B^n)}^2
        \right)
        \\
        &\aleq
        (1+\Upsilon)
        \left(
        \frac{r}{M^{\nu(j)}\Gamma}
        +\brac{
        \frac{r}{M^{\nu(j)}\Gamma}}^2
        \right)
        \\
        &\aleq
        (1+\Upsilon)\frac{r}{M^{\nu(j)}\Gamma}.
\end{split}
\]
Here we used also that by  \eqref{eq:severalupdates1:rleqMGammadef} we have $\frac{r}{M\Gamma} \leq 1$, so $\frac{r}{M^{\nu(j)} \Gamma} \leq 1$. Thus \eqref{eq:Hjest} holds.

Also, by \eqref{eq:singleupdate:DHLinftyasldkj}, applied with
$\Gamma_j,\Upsilon_j$, we have
\[
\begin{split}
        \|DH_j\|_{L^\infty(\B^n)}
        \aleq
        (1+\Upsilon_j)\Big(
        &L_j^{1/2}\|A_j\|_{L^\infty}
        +
        L_j^{-1/2}\|DA_j\|_{L^\infty}
        \\
        &+
        \Gamma_jL_j^{-1/2}\|A_j\|_{L^\infty}
        +
        L_j^{-1}\|A_j\|_{L^\infty}\|DA_j\|_{L^\infty}
        +
        \|A_j\|_{L^\infty}^2
        \Big).
\end{split}
\]
The five terms inside the parentheses are estimated as follows:
\[
        L_j^{1/2}\|A_j\|_{L^\infty}
        \aleq
        M^{\nu(j)}\Gamma,
\]
\[
        L_j^{-1/2}\|DA_j\|_{L^\infty}
        \aleq
        M^{1-\nu(j)}\Gamma
        \aleq
        M^{\nu(j)}\Gamma,
\]
\[
        \Gamma_jL_j^{-1/2}\|A_j\|_{L^\infty}
        \aleq
        M^{\nu(j)-1}\Gamma
        \frac{r}{M^{\nu(j)}\Gamma}
        =
        \frac{r}{M}
        \le
        \Gamma
        \le
        M^{\nu(j)}\Gamma,
\]
\[
\begin{split}
        L_j^{-1}\|A_j\|_{L^\infty}\|DA_j\|_{L^\infty}
        &\aleq
        \frac{r}{M^{2\nu(j)}\Gamma^2}
        r^{1/2}
        \frac{M\Gamma^2}{r^{1/2}}
        \\
        &=
        rM^{1-2\nu(j)}
        \le
        M^{2-2\nu(j)}\Gamma
        \le
        M^{\nu(j)}\Gamma,
\end{split}
\]
and
\[
        \|A_j\|_{L^\infty}^2
        \aleq
        r
        \le
        M\Gamma
        \le
        M^{\nu(j)}\Gamma.
\]
Since $1+\Upsilon_j\aleq 1+\Upsilon$, we conclude that \eqref{eq:Hjestv2} holds

Now fix
$m\in\{1,\ldots,d\}$. Since every correction $H_i$ has zero coordinates outside its
assigned block $2m(i)$, $2m(i)+1$, $1$, we have
\[
        F_j^{2m}
        =
        F^{2m}
        +
        \sum_{\substack{1\leq i\leq j\\ m(i)=m}}
        H_i^{2m},
        \qquad
        F_j^{2m+1}
        =
        F^{2m+1}
        +
        \sum_{\substack{1\leq i\leq j\\ m(i)=m}}
        H_i^{2m+1}.
\]

Recall \eqref{eq:severalupdates1:bookkkeeperkappadef}. By \eqref{eq:severalupdates1:DF2mLinftyest} and
\eqref{eq:DHj2mjest23f},
\[
\begin{split}
        \|DF_j^{2m}\|_{L^\infty(\B^n)}
        +
        \|DF_j^{2m+1}\|_{L^\infty(\B^n)}
        &\aleq
        \Gamma
        +
        \sum_{\substack{1\leq i\leq j\\ m(i)=m}}
        M^{\nu(i)}\Gamma
        \\
        &\aleq
        C_jM^{\kappa_j(m)}\Gamma,
\end{split}
\]
This proves
\eqref{eq:severalupdates1:inductionDFj2mLinfty} at level $j$.

Similarly, by \eqref{eq:severalupdates1:D2F2mLinftyestD2ver} and
\eqref{eq:D2Hjest},
\[
\begin{split}
        \|D^2F_j^{2m}\|_{L^\infty(\B^n)}
        +
        \|D^2F_j^{2m+1}\|_{L^\infty(\B^n)}
        &\aleq
        \frac{\Gamma^3}{r}
        +
        \sum_{\substack{1\leq i\leq j\\ m(i)=m}}
        \frac{M^{3\nu(i)}\Gamma^3}{r}
        \\
        &\aleq\frac{M^{3\kappa_j(m)}\Gamma^3}{r}.
\end{split}
\]
This proves \eqref{eq:severalupdates1:inductionD2Fj2mLinfty} at level $j$.

Finally, by \eqref{eq:severalupdates1:F2mlinftyupsilonasd},
\eqref{eq:changeHjC0}, and using that by \eqref{eq:severalupdates1:rleqMGammadef} $r\leq M\Gamma$,
\[
\begin{split}
        \|F_j^{2m}\|_{L^\infty(\B^n)}
        &\le
        \Upsilon
        +
        \sum_{\substack{1\leq i\leq j\\ m(i)=m}}
        \|H_i^{2m}\|_{L^\infty(\B^n)}
        \\
        &\aleq
        \Upsilon
        +
        \sum_{\substack{1\leq i\leq j\\ m(i)=m}}
        \frac{r}{M^{\nu(i)}\Gamma}
        \\
        &\aleq
        \Upsilon+\frac{r}{M\Gamma}
        \aleq
        1+\Upsilon.
\end{split}
\]

This proves \eqref{eq:severalupdates1:inductionFj2mLinftyCj} at level $j$.

\underline{Hence the induction
closes.}

Set\[
        G:=F_{n+1}.
\]
\eqref{eq:severalupdates1:outputsupp} follows from \eqref{eq:FjmFsupport}.
\eqref{eq:goal:stagepassiveunchanged} proves \eqref{eq:severalupdates1:outputnochange}.

Taking $j=n+1$ in
\eqref{eq:severalupdates1:inductionDFj2mLinfty} and
\eqref{eq:severalupdates1:inductionD2Fj2mLinfty}, and using
\eqref{eq:stagekappaleqs}, gives
\[
        \max_{1\leq m\leq d}
        \|DG^{2m}\|_{L^\infty(\B^n)}+\|DG^{2m+1}\|_{L^\infty(\B^n)}
        \aleq
        M^s\Gamma,
\]
and
\[
        \max_{1\leq m\leq d}
       \|D^2G^{2m}\|_{L^\infty(\B^n)}+\|D^2G^{2m+1}\|_{L^\infty(\B^n)}
        \aleq
        \frac{M^{3s}\Gamma^3}{r}.
\]
These are \eqref{eq:severalupdates1:outputDG2nNsGanna} and
\eqref{eq:severalupdates1:outputD2G2nNsGanna}.

Summing
\eqref{eq:stageREMstep} over $j=1,\ldots,n+1$ and using
\eqref{eq:stage-principal-cancellation}, we obtain
\begin{equation}\label{eq:GisFplusEjaslkdj}
        G^\ast\lambda
        =
        F^\ast\lambda-\omega^\ell
        +
        \sum_{j=1}^{n+1}E_j.
\end{equation}

By  \eqref{eq:goal:Ejest1}
using also \eqref{eq:severalupdates1:mollyerrorasd},
\[
        \|G^\ast\lambda\|_{L^\infty(\B^n)}
        \aleq
        \frac{r}{M}.
\]
This proves \eqref{eq:severalupdates1:outputGastlambda}.

Also, by \eqref{eq:severalupdates1:Dfastlambda},
\eqref{eq:severalupdates1:mollyeestD2as}, and
\eqref{eq:severalupdates1:elldef},
\[
\begin{split}
        \|D(F^\ast\lambda-\omega^\ell)\|_{L^\infty(\B^n)}
        &\le
        \|D(F^\ast\lambda)\|_{L^\infty(\B^n)}
        +
        \|D\omega^\ell\|_{L^\infty(\B^n)}
        \\
        &\aleq
        \Gamma^2+\frac{r}{\ell}
        \\
        &\aleq
        M\Gamma^2
        \overset{s \geq 1}{\aleq}
        M^{2s}\Gamma^2.
\end{split}
\]
Combining this with \eqref{eq:GisFplusEjaslkdj} and \eqref{eq:goal:Ejest2} yields
\[
        \|D(G^\ast\lambda)\|_{L^\infty(\B^n)}
        \aleq
        M^{2s}\Gamma^2.
\]
This proves \eqref{eq:severalupdates1:outputDGastlambdasad33e}.

It remains to estimate $G-F=\sum_{j=1}^{n+1} H_j$. Since $\nu(j) \geq 1$ from \eqref{eq:Hjest} we readily find
\[
        \|G-F\|_{L^\infty(\B^n)}
        \aleq
        (1+\Upsilon)\frac{r}{M\Gamma}.
\]
This proves \eqref{eq:severalupdates1:outputGmF}.

From \eqref{eq:Hjestv2}, using that $\nu(j) \leq s$ we also find
\[
        \|D(G-F)\|_{L^\infty(\B^n)}
        \aleq
        (1+\Upsilon)\, M^s\, \Gamma.
\]
This proves \eqref{eq:severalupdates1:outputDGmF} and completes the proof.
\end{proof}

\begin{lemma}[Mollification]
\label{la:mollypfaff}
Let
\[
        \omega\in C^1(\overline{\B^n};\Lambda^1\R^n)
\]
and let
\[
        \|\omega\|_{L^\infty(\B^n)}\leq r,
        \qquad
        \|D\omega\|_{L^\infty(\B^n)}\leq \Gamma^2.
\]
Let $M\ge1$ and set
\begin{equation}\label{eq:absolute-mollification-ell}
        \ell:=\frac{r}{M\Gamma^2}.
\end{equation}
Then there is a smooth one-form $\omega^\ell$ on $\overline{\B^n}$ such that
\begin{equation}\label{eq:absolute-mollification-error}
        \|\omega-\omega^\ell\|_{L^\infty(\B^n)}
        \le
        C\frac rM,
\end{equation}
and 
\[
 \|\omega_{\ell}\|_{L^\infty(\B^n)}\leq C \|\omega\|_{L^\infty(\B^n)}
\]
and 
\begin{equation}\label{eq:absolute-mollification-derivatives}
        \|D^a\omega^\ell\|_{L^\infty(\B^n)}
        \le
        C r\ell^{-a},
        \qquad
        a=0,1,2.
\end{equation}
Here $C$ depends only on $n$.
\end{lemma}

\begin{proof}
Any function $f: \B^n \to \R$ can be extended to a map $\widetilde{f} : \R^n \to \R$, $\widetilde{f}(x) := f(x/|x|)$ for $|x| \geq 1$ such that
\[
\|\widetilde{f}\|_{L^\infty(\R^n)}\leq \|f\|_{L^\infty(\B^n)}
\]
\[
\|D\widetilde{f}\|_{L^\infty(\R^n)} \aleq \|\abs{Df(x/|x|)} \frac{1}{|x|}\|_{L^\infty(\R^n\setminus \B^n)} + \|Df\|_{L^\infty(\B^n)} \leq 2 \|Df\|_{L^\infty(\B^n)}.
\]
Thus, we can extend $\omega$ coefficient-wise to a Lipschitz one-form $\widetilde\omega \in \lip(\R^N,\Ep^1\R^N)$ on
$\R^n$, and with dimensional constants $C>0$
\[
        \|\widetilde\omega\|_{L^\infty(\R^n)}
        \le
        C\|\omega\|_{L^\infty(\B^n)},
        \qquad
        \|D\widetilde\omega\|_{L^\infty(\R^n)}
        \le
        C\|D\omega\|_{L^\infty(\B^n)}.
\]
Let $\rho$ be a standard mollifier and set $\omega^\ell:=\rho_\ell*\widetilde\omega$.
Then
\[
        \|\omega-\omega^\ell\|_{L^\infty(\B^n)}
        \le
        C\ell\|D\omega\|_{L^\infty(\B^n)}
        \le
        C\frac rM,
\]
which proves \eqref{eq:absolute-mollification-error}. Also,
\[
        \|\omega^\ell\|_{L^\infty(\B^n)} \aleq \|\omega\|_{L^\infty}\leq r
\]
and
\[
        \|D^2\omega^\ell\|_{L^\infty(\B^n)}
        \le
        C\ell^{-2}\|\widetilde\omega\|_{L^\infty(\R^n)}
        \le
        C r\ell^{-2}.
\]
Lastly,
\[
        \|D\omega^\ell\|_{L^\infty(\B^n)}
        \le
        C\|D\widetilde\omega\|_{L^\infty(\R^n)}
        \le
        C\Gamma^2
        \le
        C M\Gamma^2
        =
        C r\ell^{-1},
\]
Thus \eqref{eq:absolute-mollification-derivatives} follows.
\end{proof}

Now we adapt the ideas from \Cref{pr:severalupdates1} above to the situation when we cannot assume that $F^\ast\lambda = 0$ close to $\partial \B^n$. Moreover, observe that we introduce a new parameter that restricts which coordinates will be changed: $b \leq d$.
\begin{proposition}
\label{pr:keepcoefficientsargnobd}
Let $b\in\{1,\ldots,d\}$ and set $s:=\left\lceil\frac{n+1}{b}\right\rceil$.

There exist constants $\eps_{0,b}>0$, $C_{\mathrm{st},b}\ge1$, depending only on $n,d,N,b$  with the following property.

Let
\[
        F\in C^\infty(\overline{\B^n};\R^N),
        \qquad
        M\ge1,
        \qquad
        \Gamma\ge1,
        \qquad
        \Upsilon\ge0,
        \qquad
        r>0.
\]
Assume \eqref{eq:severalupdates1:rleqMGammadef} without the $\ell$-condition, i.e.
\begin{equation}\label{eq:vs2propassaslkdj-scale}
        r\leq M\Gamma,
        \qquad
        r\leq M^2\Gamma^2,
\end{equation}
and \eqref{eq:severalupdates1:Fasttlambdas0fdi2}, i.e.
\[        \|F^\ast\lambda\|_{L^\infty(\B^n)}
        \le
        \eps_{0,b}r,
\]
and \eqref{eq:severalupdates1:DF2mLinftyest} but only up to order $b$, i.e.
\begin{equation}\label{eq:vs2propassaslkdj-first}
        \max_{1\leq m\leq b}
        \|D(F^{2m},F^{2m+1})\|_{L^\infty(\B^n)}
        \le
        \Gamma,
\end{equation}

and \eqref{eq:severalupdates1:D2F2mLinftyestD2ver} also only up to order $b$, i.e.
\begin{equation}\label{eq:vs2propassaslkdj-second}
        \max_{1\leq m\leq b}
        \|D^2(F^{2m},F^{2m+1})\|_{L^\infty(\B^n)}
        \le
        \frac{\Gamma^3}{r},
\end{equation}
and \eqref{eq:severalupdates1:Dfastlambda}, i.e.
\[
        \|D(F^\ast\lambda)\|_{L^\infty(\B^n)}
        \le
        \Gamma^2,
\]
and \eqref{eq:severalupdates1:F2mlinftyupsilonasd} again only up to order $b$, i.e.
\begin{equation}\label{eq:vs2propassaslkdj-Upsilon}
        \max_{1\leq m\leq b}
        \|F^{2m}\|_{L^\infty(\B^n)}
        \le
        \Upsilon.
\end{equation}
Then there exists $G\in C^\infty(\overline{\B^n};\R^N)$
such that \eqref{eq:severalupdates1:outputGastlambda} holds, i.e.
\[
        \|G^\ast\lambda\|_{L^\infty(\B^n)}
        \le
        C_{\mathrm{st},b}\frac rM,
\]
and we have \eqref{eq:severalupdates1:outputGmF}
i.e.
\[
        \|G-F\|_{L^\infty(\B^n)}
        \le
        C_{\mathrm{st},b}(1+\Upsilon)\frac{r}{M\Gamma},
\]
and \eqref{eq:severalupdates1:outputDGmF}, i.e.
\[
        \|D(G-F)\|_{L^\infty(\B^n)}
        \le
        C_{\mathrm{st},b}(1+\Upsilon)M^s\Gamma,
\]
and \eqref{eq:severalupdates1:outputDG2nNsGanna} for the first $b$ coordinates i.e.
\begin{equation}\label{eq:vs2propassaslkdj-output-first}
        \max_{1\leq m\leq b}
        \|D(G^{2m},G^{2m+1})\|_{L^\infty(\B^n)}
        \le
        C_{\mathrm{st},b}M^s\Gamma,
\end{equation}
and similarly \eqref{eq:severalupdates1:outputD2G2nNsGanna} for the first $b$ coordinates.
\begin{equation}\label{eq:vs2propassaslkdj-output-second}
        \max_{1\leq m\leq b}
        \|D^2(G^{2m},G^{2m+1})\|_{L^\infty(\B^n)}
        \le
        C_{\mathrm{st},b}\frac{M^{3s}\Gamma^3}{r},
\end{equation}
and we have again \eqref{eq:severalupdates1:outputDGastlambdasad33e} i.e.
\[        \|D(G^\ast\lambda)\|_{L^\infty(\B^n)}
        \le
        C_{\mathrm{st},b}M^{2s}\Gamma^2,
\]
but this time the following coordinates remain unchanged (cf. \eqref{eq:severalupdates1:outputnochange})
\begin{equation}\label{eq:vs2propassaslkdj-output-frozen}
        G^j=F^j
        \qquad\text{for every }j=2b+2,\ldots,N.
\end{equation}
\end{proposition}
\begin{proof}
We choose $\ell$ as in \eqref{eq:severalupdates1:elldef}, observe that we don't need to establish \eqref{eq:severalupdates1:outputsupp}.
Let
\[
        \omega:=F^\ast\lambda,
        \qquad
        \ell:=\frac{r}{M\Gamma^2}.
\]
Apply \Cref{la:mollypfaff} to $\omega$ and obtain
$\omega^\ell$ with \eqref{eq:severalupdates1:mollyerrorasd} and \eqref{eq:severalupdates1:mollyeestD2as}, i.e.
\[
        \|F^\ast \lambda-\omega^\ell\|_{L^\infty(\B^n)}
        \aleq
        \ell\|D(F^\ast \lambda)\|_{L^\infty(\B^n)}
        \aleq
        \frac{r}{M}.
\]
and
\[      \|D^a\omega^\ell\|_{L^\infty(\B^n)}
        \aleq
        r\ell^{-a},
        \qquad
        a=0,1,2.
\]
Again we have from \eqref{eq:severalupdates1:Fasttlambdas0fdi2}  $\left \|\frac{\omega^\ell}{r} \right \|_{L^\infty(\B^n)} \leq \eps_0$, and thus, for $\eps_0$ small enough, $-\omega^\ell/r$ belongs to the ball $B_{\eps^\ast}(0)$ from
\Cref{pr:geometriclemma}.

The following is \eqref{eq:severalupdates1:Ajdef} for $\theta \equiv 1$.
\begin{equation}\label{eq:vs2propassaslkdj-amplitudes}
        A_j
        :=
        r^{1/2}
        \gamma_j\left(-\frac{\omega^\ell}{r}\right),
        \qquad
        j=1,\ldots,n+1.
\end{equation}
Again, \Cref{pr:geometriclemma} gives \eqref{eq:stage-principal-cancellation}, i.e.
\[
        -\omega^\ell
        =
        \sum_{j=1}^{n+1}A_j^2\,d\phi_j
        \qquad\text{in }\B^n.
\]
The computation of
\eqref{eq:severalupdates1:AestLinftycrsqert}, \eqref{eq:severalupdates1:DAestLinfty} \eqref{eq:severalupdates1:D2Aasdx} applies verbatim,
with $\theta\equiv1$.

Let $m_b(j)$ and $\nu_b(j)$ be from \Cref{la:wigglebookkeeper} i.e. as in
\eqref{eq:bookkeeper:asdopasd2def}. Define
\begin{equation}\label{eq:vs2propassaslkdj-frequency}
        L_j:=\frac{M^{2\nu_b(j)}\Gamma^2}{r},
        \qquad
        j=1,\ldots,n+1.
\end{equation}
which is essentially the same as \eqref{eq:severalupdates1:Lfrequencydef}, just with $d$ replaced by $b$.

By \eqref{eq:vs2propassaslkdj-scale}, $L_j\ge1$.

Set $F_0:=F$. For $j=0,\ldots,n+1$ and $m=1,\ldots,b$, define
\begin{equation}\label{eq:vs2propassaslkdj-kappa}
        \kappa_j(m)
        :=
        \max\bigl(\{0\}\cup
        \{\nu_b(i):1\leq i\leq j,\ m_b(i)=m\}\bigr).
\end{equation}
By \Cref{la:wigglebookkeeper},
\begin{equation}\label{eq:vs2propassaslkdj-kappa-before}
        \kappa_{j-1}(m_b(j))\leq \nu_b(j)-1,
        \qquad
        j=1,\ldots,n+1,
\end{equation}
and
\begin{equation}\label{eq:vs2propassaslkdj-kappa-upper}
        \kappa_j(m)\leq s
        \qquad\text{for all }j,m.
\end{equation}

We now repeat the finite induction from the proof of
\Cref{pr:severalupdates1}, essentially verbatim just replacing $d$ by $b$ and using our $\kappa_j$ and $s$.

Namely, we construct $F_j$ inductively so that, for every $j=0,\ldots,n+1$

\begin{itemize} \item \eqref{eq:severalupdates1:inductionDFj2mLinfty} up to $b$: For every
$m\in\{1,\ldots,b\}$,
\[
        \|DF_j^{2m}\|_{L^\infty(\B^n)}
        +
        \|DF_j^{2m+1}\|_{L^\infty(\B^n)}
        \le
        C_j M^{\kappa_j(m)}\Gamma,
\]
\[        \|D^2F_j^{2m}\|_{L^\infty(\B^n)}
        +
        \|D^2F_j^{2m+1}\|_{L^\infty(\B^n)}
        \le
        C_j\frac{M^{3\kappa_j(m)}\Gamma^3}{r},
\]
and
\[
        \|F_j^{2m}\|_{L^\infty(\B^n)}
        \le
        C_j(1+\Upsilon).
\]
Here $C_j$ is a constant that will change with $j$ but be a product of uniform constants.

\item we do not care about \eqref{eq:FjmFsupport}
\item We have \eqref{eq:goal:stagepassiveunchanged} but for more coordinates
\begin{equation}\label{eq:goal:stagepassiveunchangedwithb}
        F_j^k=F^k
        \qquad\text{for }k=2b+2,\ldots,N.
\end{equation}
\item If $j \geq 1$, \eqref{eq:stageREMstep} with estimates \eqref{eq:goal:Ejest1}, \eqref{eq:goal:Ejest2}.
\item Estimates \eqref{eq:Hjest}, \eqref{eq:Hjestv2}
\end{itemize}
The induction argument is verbatim, one just ignores the support properties and replaces $d$ by $b$ -- and then the conclusion is verbatim as well.
\end{proof}

\section{Extending horizontal maps: Proof of Theorem~\ref{th:lambdaext}}\label{s:extension}
Also throughout this section, $\lambda$ is the fixed form from \eqref{eq:pfaffcr}. In view of \Cref{th:pfaffconstantrank} for \Cref{th:lambdaext} we need to prove the following
\begin{theorem}
\label{th:pfaffextension}
Let $N\ge 2d+1$, $d\ge1$, $n\in\{2,3,\ldots\}$.
For $s:=\left\lceil\frac{n+1}{d}\right\rceil$ assume
\begin{equation}\label{eq:alphasdoiuvc}
        \frac12<\alpha<\frac{s+1}{2s+1}.
\end{equation}
For any $F_0\in C^\infty(\overline{\B^n};\R^N)$ such that $f:= F_0\Big |_{\partial \B^n}: \partial \B^n \to \R^N$ satisfies for $\lambda$ from \eqref{eq:pfaffcr}
\[
        f^\ast\lambda=0  \quad \text{on $\partial\B^n$}
\]
there exists $F\in C^\alpha(\overline{\B^n};\R^N)$
such that
\begin{equation}\label{eq:pfaffext:horizontalboundaryFF0}
        F=F_0
        \qquad\text{on }\partial\B^n,
\end{equation}
and
\[
        F^\ast\lambda=0
\]
in the distributional sense in (the closed) $\overline{\B^n}$. Moreover, for any $\beta \in (0,1/2)$ and any $\eps > 0$ we may assume
\begin{equation}\label{eq:FF0closeby}
 \|F-F_0\|_{C^\beta(\B^n)} < \eps.
\end{equation}
\end{theorem}

First we want to assume w.l.o.g. $F_0^\ast \lambda \equiv 0$ around $\partial \B^n$, which we can assume by the usual radial extension
\begin{lemma}
\label{la:smooth-horizontal-collar-replacement}
Let $F_0\in C^\infty(\overline{\B^n};\R^N)$ and let $f:=F_0|_{\partial\B^n}$. Assume
\[
        f^\ast\lambda=0
        \qquad\text{on }\partial\B^n .
\]
Then, for any $0<\beta<1$ and every $\varepsilon>0$, there exists $\widetilde F_0\in C^\infty(\overline{\B^n};\R^N)$ such that
\begin{equation}\label{eq:collar-replacement-same-boundary}
        \widetilde F_0=F_0
        \qquad\text{on }\partial\B^n,
\end{equation}
\begin{equation}\label{eq:collar-replacement-horizontal-collar}
        \widetilde F_0^\ast\lambda=0
        \qquad\text{in a neighborhood of }\partial\B^n,
\end{equation}
and
\begin{equation}\label{eq:collar-replacement-Cbeta-close}
        \|\widetilde F_0-F_0\|_{C^\beta(\overline{\B^n};\R^N)}
        <\varepsilon .
\end{equation}
\end{lemma}

\begin{proof}[Proof of \Cref{th:pfaffextension}]
Recall
\[
        s=\left\lceil\frac{n+1}{d}\right\rceil .
\]
Choose $\sigma$ such that
\begin{equation}\label{eq:thpfaffext:sigma-choice}
        \frac{2\alpha-1}{1-\alpha}
        <
        \sigma
        <
        \frac1s .
\end{equation}
This is possible by \eqref{eq:alphasdoiuvc}, since
\[
        \frac{2\alpha-1}{1-\alpha}<\frac1s
        \qquad\Longleftrightarrow\qquad
        \alpha<\frac{s+1}{2s+1}.
\]
Choose a small $\gamma>0$ such that
\begin{equation}\label{eq:thpfaffext:gamma-choice}
        \tau_\alpha
        :=
        -\brac{2\alpha-1-\sigma(1-\alpha)
        +
        \gamma\bigl((s+1)\alpha-1\bigr)}
        >0 .
\end{equation}
and, since $\beta < 1/2$, we can also ensure
\begin{equation}\label{eq:thpfaffext:gamma-choice-beta}
        \tau_\beta
        :=
        \sigma(1-\beta)-(2\beta-1)
        -
        \gamma\bigl((s+1)\beta-1\bigr)
        >0 ,
\end{equation}
and
\begin{equation}\label{eq:thpfaffext:mu-beta-choice}
        \mu_\beta
        :=
        1-2\beta
        -
        \gamma\bigl((s+1)\beta-1\bigr) = \tau_\beta-\sigma(1-\beta)
        >0 .
\end{equation}

Let $\eps_0$ be the smallness constant in
\Cref{pr:severalupdates1}. Let $C_{\mathrm{st}}\ge1$ be a
constant for which all estimates in
\eqref{eq:severalupdates1:outputGastlambda},
\eqref{eq:severalupdates1:outputGmF},
\eqref{eq:severalupdates1:outputDGmF},
\eqref{eq:severalupdates1:outputDG2nNsGanna},
\eqref{eq:severalupdates1:outputD2G2nNsGanna}, and
\eqref{eq:severalupdates1:outputDGastlambdasad33e} hold. Set
\begin{equation}\label{eq:thpfaffextv2:C0-choice}
        C_0
        :=
        \max\left\{
        2,\,
        \frac{C_{\mathrm{st}}}{\eps_0},\,
        C_{\mathrm{st}},\,
        C_{\mathrm{st}}^{1/2}
        \right\}.
\end{equation}

We may assume $F_0^\ast \lambda \not \equiv 0$, otherwise there is nothing to show.

Set
\begin{equation}\label{eq:thpfaffextv2:delta-choice}
        \delta
        :=
        \frac18
        \min\left\{
        1,\,
        \dist(\operatorname{supp}F_0^\ast\lambda,\partial \B^n)
        \right\}.
\end{equation}

By \Cref{la:smooth-horizontal-collar-replacement} we may assume $\delta > 0$.

Then
\begin{equation}\label{eq:thpfaffextv2:initial-support-away}
        \operatorname{supp}F_0^\ast\lambda
        \subset
        \{x\in \B^n:\dist(x,\partial \B^n)>4\delta\}.
\end{equation}

For $k=0,1,2,\ldots$, define
\begin{equation}\label{eq:thpfaffextv2:Omega-k}
        \Omega_k
        =
        \{x\in \B^n:\dist(x,\partial \B^n)>2\delta+2^{-k}\delta\}.
\end{equation}
In particular $F_0^\ast \lambda \equiv 0$ for all $x \not \in \Omega_k$, for all $k$.
Then
\begin{equation}\label{eq:thpfaffextv2:Omega-nesting}
        \overline{\Omega_k}\subset\Omega_{k+1}
        \qquad\text{for every }k\ge0,
\end{equation}
\begin{equation}\label{eq:thpfaffextv2:Omega-union}
        \bigcup_{k=0}^\infty\Omega_k
        \subset
        \{x\in \B^n:\dist(x,\partial \B^n)>2\delta\},
\end{equation}
and
\begin{equation}\label{eq:thpfaffextv2:gap-lower-bound}
        \dist(\Omega_k,\B^n\setminus\Omega_{k+1})
        \ge
        2^{-k-1}\delta
        \qquad\text{for every }k\ge0.
\end{equation}

Set
\begin{equation}\label{eq:thpfaffextv2:Upsilon-star}
        \Upsilon_\ast
        :=
        1+
        \max_{1\leq m\leq d}
        \|F_0^{2m}\|_{L^\infty(\B^n)} .
\end{equation}
Set
\begin{equation}\label{eq:thpfaffextv2:r0-choice}
        r_0
        :=
        \max\left\{
        1,\,
        \eps_0^{-1}
        \|F_0^\ast\lambda\|_{L^\infty(\B^n)}
        \right\}.
\end{equation}

Let $\Sigma_\beta \ge1$ be a constant, depending only on
$n,d,N,\beta,C_{\mathrm{st}}$, large enough for the $C^\beta$ interpolation
estimate in \eqref{eq:thpfaffext:Cbeta-increment} below.

Set
\begin{equation}\label{eq:thpfaffextv2:Gamma0-choice}
\begin{split}
        \Gamma_0
        :=
        \max\Bigg\{&
        1,\,
        \max_{1\leq m\leq d}
        \|D(F_0^{2m},F_0^{2m+1})\|_{L^\infty(\B^n)},\\
        &\left(
        r_0
        \max_{1\leq m\leq d}
        \|D^2(F_0^{2m},F_0^{2m+1})\|_{L^\infty(\B^n)}
        \right)^{1/3},\\
        &\|D(F_0^\ast\lambda)\|_{L^\infty(\B^n)}^{1/2},\,
        r_0^{1/(1+\gamma)},\,
        (2C_0)^{1/\gamma},\\
        &\brac{2 C_0^{-(1+\gamma)} }^{\frac{1}{\gamma{1+s(\gamma+2)}}},\brac{2 C_0^{-\gamma}}^{\frac{1}{\gamma(1+s(\gamma+1))}}\\
        &C_0^{(1+\sigma)/(\gamma(1-s\sigma))},\,
        \left(\frac{8r_0}{\delta}\right)^{1/(\gamma+2)},\\
        &2^{1/(\gamma(1+s(\gamma+2)))},\,
        2^{1/(\gamma(1+s(\gamma+1)))},\,
        2^{1/(\tau_\alpha s\gamma)},\\
        &\bigl(2C_{\mathrm{st}}(1+\Upsilon_\ast)r_0\bigr)^{1/(\gamma+1)}\\,
        &2^{1/(\tau_\beta s\gamma)},
        \left(
        \frac{4\Sigma_\beta(1+\Upsilon_\ast)r_0}{\eps}
        \right)^{1/(1+\gamma)},
        \left(
        \frac{4\Sigma_\beta(1+\Upsilon_\ast)r_0^{1-\beta}}{\eps}
        \right)^{1/\mu_\beta},2^{1/(\tau_\alpha s\gamma)},\, 2^{1/\tau_\alpha}
        \Bigg\}.
\end{split}
\end{equation}

For $k=0,1,2,\ldots$, define recursively
\begin{equation}\label{eq:thpfaffextv2:parameters}
\begin{split}
        M_k&:=\Gamma_k^\gamma \xrightarrow{k \to \infty} \infty\\
        r_{k+1}&:=C_0\frac{r_k}{M_k} \xrightarrow{k \to \infty} 0,\\
        \Gamma_{k+1}&:=C_0M_k^s\Gamma_k
        =
        C_0\Gamma_k^{1+s\gamma} \xrightarrow{k \to \infty} \infty.
\end{split}
\end{equation}
Also set
\begin{equation}\label{eq:thpfaffextv2:Lambda-ell}
\begin{split}
        \Lambda_k&:=\frac{\Gamma_k^3}{r_k},\\
        \ell_k&:=\frac{r_k}{M_k\Gamma_k^2} \equiv \frac{r_k}{\Gamma_k^{2+\gamma}}.
\end{split}
\end{equation}

By \eqref{eq:thpfaffextv2:Gamma0-choice}, we have
\begin{equation}\label{eq:thpfaffextv2:Gamma0-first}
        \max_{1\leq m\leq d}
        \|D(F_0^{2m},F_0^{2m+1})\|_{L^\infty(\B^n)}
        \le
        \Gamma_0,
\end{equation}
\begin{equation}\label{eq:thpfaffextv2:Gamma0-second}
        \max_{1\leq m\leq d}
        \|D^2(F_0^{2m},F_0^{2m+1})\|_{L^\infty(\B^n)}
        \le
        \frac{\Gamma_0^3}{r_0}
        =
        \Lambda_0,
\end{equation}
\begin{equation}\label{eq:thpfaffextv2:Gamma0DFastlambda}
        \|D(F_0^\ast\lambda)\|_{L^\infty(\B^n)}
        \le
        \Gamma_0^2,
\end{equation}
\begin{equation}\label{eq:thpfaffextv2:Gamma0scaling}
        r_0\le\Gamma_0^{1+\gamma}=M_0\Gamma_0,
\end{equation}
\begin{equation}\label{eq:thpfaffextv2:M0-large}
        M_0=\Gamma_0^\gamma\ge2C_0,
\end{equation}
and
\begin{equation}\label{eq:thpfaffextv2:scale-relation}
        C_0^{1+\sigma}\Gamma_0^{-\gamma(1-s\sigma)}
        \leq 1.
\end{equation}

Since $C_0\ge1$, $\Gamma_0\ge1$, and
$\Gamma_{k+1}=C_0\Gamma_k^{1+s\gamma}$, we have $\Gamma_k\ge\Gamma_0$ for every
$k\ge0$. Moreover,

\begin{equation}\label{eq:thpfaffextv2:ell-ratio}
\begin{split}
        \frac{\ell_{k+1}}{\ell_k}
        &=
        \frac{r_{k+1}}{M_{k+1}\Gamma_{k+1}^2}
        \frac{M_k\Gamma_k^2}{r_k}\\
        &=C_0
        \frac{1}{\brac{C_0 \Gamma_k^{1+s\gamma}}^{2+\gamma}}
         \Gamma_k^2\\
        &=
        C_0^{-(1+\gamma)}
        \Gamma_k^{-\gamma(1+s(\gamma+2))}\\
        &\leq
        C_0^{-(1+\gamma)}
        \Gamma_0^{-\gamma(1+s(\gamma+2))}\\
        &\le
        \frac12.
\end{split}
\end{equation}

Also,
\begin{equation}\label{eq:thpfaffextv2:ell0-gap}
        4\ell_0
        =
        4\frac{r_0}{\Gamma_0^{\gamma+2}}
        \le
        \frac{\delta}{2}.
\end{equation}
Therefore \eqref{eq:thpfaffextv2:ell-ratio} and
\eqref{eq:thpfaffextv2:ell0-gap} give
\begin{equation}\label{eq:thpfaffextv2:gap-condition}
        4\ell_k
        \le
        2^{-k-1}\delta
        \le
        \dist(\Omega_k,\B^n\setminus\Omega_{k+1})
        \qquad\text{for every }k\ge0.
\end{equation}

Furthermore, as in \eqref{eq:thpfaffextv2:ell-ratio}
\begin{equation}\label{eq:thpfaffextv2:C0-summability-ratio}
\begin{split}
        \frac{r_{k+1}}{M_{k+1}\Gamma_{k+1}}
        \frac{M_k\Gamma_k}{r_k}=&
        \frac{r_{k+1}}{M_{k+1}\Gamma_{k+1}^2}
        \frac{M_k\Gamma_k^2}{r_k} \frac{\Gamma_{k+1}}{\Gamma_k}
        \\
        =&C_0^{-(1+\gamma)}
        \Gamma_k^{-\gamma(1+s(\gamma+2))} \frac{C_0 \Gamma_k^{1+s\gamma}}{\Gamma_k}\\
        =&C_0^{-\gamma}
        \Gamma_k^{-\gamma(1+s(\gamma+1))}\leq C_0^{-\gamma}
        \Gamma_0^{-\gamma(1+s(\gamma+1))}\\
        &\le
        \frac12.
\end{split}
\end{equation}

Since
\begin{equation}\label{eq:thpfaffextv2:C0-summability-first-term}
        \frac{r_0}{M_0\Gamma_0}
        =
        \frac{r_0}{\Gamma_0^{\gamma+1}}
        \le
        \frac12,
\end{equation}
we obtain
\[
        \sum_{k=0}^\infty
        \frac{r_k}{M_k\Gamma_k}
        \le
        \sum_{k=0}^\infty
        2^{-k}
        \frac{r_0}{M_0\Gamma_0} =2\frac{r_0}{\Gamma_0^{1+\gamma}}
        \leq 1.
\]
and thus by \eqref{eq:thpfaffextv2:Gamma0-choice}
\begin{equation}\label{eq:thpfaffextv2:C0-summability}
\begin{split}
        \sum_{k=0}^\infty
        \frac{r_k}{M_k\Gamma_k}
        &\le
        2\frac{r_0}{\Gamma_0^{1+\gamma}},\\
        C_{\mathrm{st}}(1+\Upsilon_\ast)
        \sum_{k=0}^\infty
        \frac{r_k}{M_k\Gamma_k}
        &\le
        2C_{\mathrm{st}}(1+\Upsilon_\ast)
        \frac{r_0}{\Gamma_0^{1+\gamma}}
        \le1.
\end{split}
\end{equation}

We also observe, since $\tau_\alpha>0$ and using \eqref{eq:thpfaffextv2:Gamma0-choice}
\[
        \frac{\Gamma_{k+1}^{-\tau_\alpha}}{\Gamma_k^{{-\tau_\alpha}}}
        \overset{\eqref{eq:thpfaffextv2:parameters}}{=}
        \brac{C_0 \Gamma_k^{1+s\gamma}}^{-\tau_\alpha} \Gamma_k^{\tau_\alpha}
        =
        C_0^{{-\tau_\alpha}}\Gamma_k^{{-\tau_\alpha} s\gamma}
        \le
        \Gamma_0^{{-\tau_\alpha} s\gamma}
        \le
        \frac12.
\]
and thus
\begin{equation}\label{eq:thpfaffext:Calpha-ratiosummable}
 \sum_{k=0}^\infty \Gamma_{k}^{-\tau_\alpha}\leq \sum_{k=0}^\infty 2^{-k} \Gamma_{0}^{-\tau_\alpha} = 2 \Gamma_{0}^{-\tau_\alpha} \leq 1.
\end{equation}

Similarly, since $\tau_\beta > 0$ by \eqref{eq:thpfaffextv2:Gamma0-choice} and
\eqref{eq:thpfaffext:gamma-choice-beta},
\[
        \frac{\Gamma_{k+1}^{-\tau_\beta}}{\Gamma_k^{-\tau_\beta}}
        =
        C_0^{-\tau_\beta}
        \Gamma_k^{-\tau_\beta s\gamma}
        \le
        \Gamma_0^{-\tau_\beta s\gamma}
        \le
        \frac12 .
\]
Hence we record
\begin{equation}\label{eq:thpfaffext:Cbeta-ratiosummable}
        \sum_{k=0}^\infty
        \Gamma_k^{-\tau_\beta}
        \le
        2\Gamma_0^{-\tau_\beta}.
\end{equation}

Now we claim that
\begin{equation}\label{eq:thpfaffextv2:rk-Gammak-relation}
        r_k\leq r_0\Gamma_0^\sigma\Gamma_k^{-\sigma}
        \qquad\text{for every }k\ge0.
\end{equation}
For $k=0$, this is equality. If \eqref{eq:thpfaffextv2:rk-Gammak-relation}
holds at level $k$, then, using \eqref{eq:thpfaffextv2:parameters},

\begin{equation*}
\begin{split}
        r_{k+1}\Gamma_{k+1}^\sigma
        &=
        C_0\frac{r_k}{M_k}
        \left(C_0M_k^s\Gamma_k\right)^\sigma\\
        &=
        C_0\frac{r_k}{\Gamma_k^\gamma}
        \left(C_0\Gamma_k^{\gamma s}\Gamma_k\right)^\sigma\\
        &=
        C_0^{1+\sigma}
        r_k\Gamma_k^\sigma
        \Gamma_k^{-\gamma(1-s\sigma)}\\
        &\overset{\eqref{eq:thpfaffextv2:rk-Gammak-relation}}{\le}
        C_0^{1+\sigma}
        r_0\Gamma_0^\sigma
        \Gamma_k^{-\gamma(1-s\sigma)}\\
        &\leq
        C_0^{1+\sigma}\Gamma_0^{-\gamma(1-s\sigma)}
        r_0\Gamma_0^\sigma
        \\
        &\le
        r_0\Gamma_0^\sigma.
\end{split}
\end{equation*}
The last inequality follows from \eqref{eq:thpfaffextv2:scale-relation}, $s\sigma <1$, and
$\Gamma_k\ge\Gamma_0\geq 1$. Hence \eqref{eq:thpfaffextv2:rk-Gammak-relation} holds for
every $k$. In particular, $r_k\to0$.

Next we claim that
\begin{equation}\label{eq:thpfaffext:rk-scale-one}
        r_k\leq M_k\Gamma_k
        \qquad\text{for every }k\ge0.
\end{equation}

For $k=0$, this is exactly \eqref{eq:thpfaffextv2:Gamma0scaling}. If
\eqref{eq:thpfaffext:rk-scale-one} holds at level $k$, then
\begin{equation*}
\begin{split}
        r_{k+1}
        &=
        C_0\frac{r_k}{M_k}
        \overset{\eqref{eq:thpfaffext:rk-scale-one}}{\le}
        C_0\Gamma_k
        \le
        C_0\Gamma_k^{1+s\gamma}
        =
        \Gamma_{k+1}
        \le
        M_{k+1}\Gamma_{k+1}.
\end{split}
\end{equation*}
Thus \eqref{eq:thpfaffext:rk-scale-one} holds for every $k$. Since
$M_k\Gamma_k\ge1$, we also have
\begin{equation}\label{eq:thpfaffext:rk-scale-two}
        r_k\leq M_k^2\Gamma_k^2
        \qquad\text{for every }k\ge0.
\end{equation}
The following consequences of the choice of $C_0$ will be used at every step:
\begin{equation}\label{eq:thpfaffext:output-compasdeps0askldj}
        C_{\mathrm{st}}\frac{r_k}{M_k}
        \le
        \eps_0 r_{k+1}.
\end{equation}
\begin{equation}\label{eq:thpfaffext:output-first-compatible}
        C_{\mathrm{st}}M_k^s\Gamma_k
        \le
        \Gamma_{k+1}.
\end{equation}
\begin{equation}\label{eq:thpfaffext:output-second-compatible}
        C_{\mathrm{st}}\frac{M_k^{3s}\Gamma_k^3}{r_k}
        \le
        \frac{\Gamma_{k+1}^3}{r_{k+1}}.
\end{equation}
\begin{equation}\label{eq:thpfaffext:output-residual-derivative-compatible}
        C_{\mathrm{st}}M_k^{2s}\Gamma_k^2
        \le
        \Gamma_{k+1}^2.
\end{equation}
Indeed, \eqref{eq:thpfaffext:output-compasdeps0askldj} follows from
$r_{k+1}=C_0r_k/M_k$ and $C_0\ge C_{\mathrm{st}}/\eps_0$ by \eqref{eq:thpfaffextv2:C0-choice}.

Similarly, \eqref{eq:thpfaffext:output-first-compatible} follows from
$\Gamma_{k+1}=C_0M_k^s\Gamma_k$ and $C_0\ge C_{\mathrm{st}}$. For
\eqref{eq:thpfaffext:output-second-compatible}, observe that
\begin{equation*}
\begin{split}
        \frac{\Gamma_{k+1}^3}{r_{k+1}}
        &=
        \frac{C_0^3M_k^{3s}\Gamma_k^3}{C_0r_k/M_k}
        =
        C_0^2M_k
        \frac{M_k^{3s}\Gamma_k^3}{r_k}
        \ge
        C_{\mathrm{st}}
        \frac{M_k^{3s}\Gamma_k^3}{r_k},
\end{split}
\end{equation*}
because $M_k\ge1$ and $C_0^2\ge C_{\mathrm{st}}$. As for
\eqref{eq:thpfaffext:output-residual-derivative-compatible},
\[
        \Gamma_{k+1}^2 = C_0^2M_k^{2s}\Gamma_k^2 \geq C_{\mathrm{st}}M_k^{2s}\Gamma_k^2
\]
since $C_0^2\ge C_{\mathrm{st}}$.
We now construct $F_k$ by induction. The induction hypotheses are the following:
\begin{itemize}
\item We have
\begin{equation}\label{eq:thpfaffext:induction-smooth}
        F_k\in C^\infty(\B^n;\R^N).
\end{equation}
For $k=0$ the following is immediate.
\item
\begin{equation}\label{eq:thpfaffext:induction-support}
        \supp F_k^\ast\lambda\subset\Omega_k.
\end{equation}
For $k=0$ this is \eqref{eq:thpfaffextv2:initial-support-away}.
\item
\begin{equation}\label{eq:thpfaffext:induction-residual}
        \|F_k^\ast\lambda\|_{L^\infty(\B^n)}
        \le
        \eps_0 r_k.
\end{equation}
For $k=0$, these follow from
\eqref{eq:thpfaffextv2:r0-choice},

\item
\begin{equation}\label{eq:thpfaffext:induction-first}
        \max_{1\leq m\leq d}
        \|D(F_k^{2m},F_k^{2m+1})\|_{L^\infty(\B^n)}
        \le
        \Gamma_k.
\end{equation}
For $k=0$ this is \eqref{eq:thpfaffextv2:Gamma0-first}.

\item
\begin{equation}\label{eq:thpfaffext:induction-second}
        \max_{1\leq m\leq d}
        \|D^2(F_k^{2m},F_k^{2m+1})\|_{L^\infty(\B^n)}
        \le
        \frac{\Gamma_k^3}{r_k}
        =
        \Lambda_k.
\end{equation}
For $k=0$ this is \eqref{eq:thpfaffextv2:Gamma0-second}.

\item
\begin{equation}\label{eq:thpfaffext:induction-residual-derivative}
        \|D(F_k^\ast\lambda)\|_{L^\infty(\B^n)}
        \le
        \Gamma_k^2.
\end{equation}
For $k=0$
\eqref{eq:thpfaffextv2:Gamma0DFastlambda}.
\item
\begin{equation}\label{eq:thpfaffext:induction-Upsilon}
        \max_{1\leq m\leq d}
        \|F_k^{2m}\|_{L^\infty(\B^n)}
        \le
        \Upsilon_\ast.
\end{equation}
For $k=0$ this is \eqref{eq:thpfaffextv2:Upsilon-star}.
\item For $k\geq 1$ we have
\begin{equation}\label{eq:thpfaffext:increment-C0}
        \|F_{k}-F_{k-1}\|_{L^\infty(\B^n)}
        \le
        C_{\mathrm{st}}(1+\Upsilon_\ast)
        \frac{r_{k-1}}{M_{k-1}\Gamma_{k-1}}.
\end{equation}
and
\begin{equation}\label{eq:thpfaffext:increment-C1}
        \|D(F_{k}-F_{k-1})\|_{L^\infty(\B^n)}
        \le
        C_{\mathrm{st}}(1+\Upsilon_\ast)M_{k-1}^s\Gamma_{k-1}.
\end{equation}
(There is nothing to show for $k=0$)

%
\end{itemize}

Assume now that $F_k$ has been constructed and satisfies
\eqref{eq:thpfaffext:induction-smooth}--\eqref{eq:thpfaffext:increment-C1}.
If $F_k^\ast\lambda\equiv0$, the construction can be stopped and the conclusion
follows with this smooth map. Thus we assume $F_k^\ast\lambda\not\equiv0$ and set
\begin{equation}\label{eq:thpfaffext:Kk-definition}
        K_k:=\supp F_k^\ast\lambda.
\end{equation}

We apply \Cref{pr:severalupdates1} with
\begin{equation*}
\begin{split}
        F=F_k,\quad
        r=r_k,\quad
        M=M_k,\quad
        \Gamma=\Gamma_k,\quad
        \Upsilon=\Upsilon_\ast,\quad
        K=K_k.
\end{split}
\end{equation*}

We check the assumptions of \Cref{pr:severalupdates1} item by
item.

\begin{itemize}
\item The smoothness assumption on $F$ holds by
\eqref{eq:thpfaffext:induction-smooth}.

\item The compact support assumption
$\supp F^\ast\lambda\Subset \B^n$ holds because, by
\eqref{eq:thpfaffext:induction-support},
\[
        K_k=\supp F_k^\ast\lambda\subset\Omega_k\Subset \B^n.
\]

\item The definition \eqref{eq:severalupdates1:elldef} becomes exactly
\[
        \ell=\frac{r_k}{M_k\Gamma_k^2}=\ell_k,
\]
which is \eqref{eq:thpfaffextv2:Lambda-ell}.

\item The first two scale assumptions in
\eqref{eq:severalupdates1:rleqMGammadef} are exactly
\eqref{eq:thpfaffext:rk-scale-one} and
\eqref{eq:thpfaffext:rk-scale-two}:
\[
        r_k\leq M_k\Gamma_k,
        \qquad
        r_k\leq M_k^2\Gamma_k^2.
\]

\item For the third assumption in
\eqref{eq:severalupdates1:rleqMGammadef}, we use
$K_k\subset\Omega_k$, \eqref{eq:thpfaffext:induction-support}. Hence by \eqref{eq:thpfaffextv2:Omega-k},
\[
        \dist(K_k,\partial \B^n)
        >
        2\delta+2^{-k}\delta>2^{-k-1}\delta\overset{\eqref{eq:thpfaffextv2:gap-condition}}{\geq} 4\ell_k
\]
This establishes \eqref{eq:severalupdates1:rleqMGammadef}.

\item The residual smallness assumption \eqref{eq:severalupdates1:Fasttlambdas0fdi2}
is precisely \eqref{eq:thpfaffext:induction-residual}:
\[
        \|F_k^\ast\lambda\|_{L^\infty(\B^n)}
        \le
        \eps_0 r_k.
\]

\item The first derivative assumption
\eqref{eq:severalupdates1:DF2mLinftyest} is precisely
\eqref{eq:thpfaffext:induction-first}:
\[
        \max_{1\leq m\leq d}
        \|D(F_k^{2m},F_k^{2m+1})\|_{L^\infty(\B^n)}
        \le
        \Gamma_k.
\]

\item The second derivative assumption
\eqref{eq:severalupdates1:D2F2mLinftyestD2ver} is precisely
\eqref{eq:thpfaffext:induction-second}:
\[
        \max_{1\leq m\leq d}
        \|D^2(F_k^{2m},F_k^{2m+1})\|_{L^\infty(\B^n)}
        \le
        \frac{\Gamma_k^3}{r_k}.
\]
\item The residual derivative assumption
\eqref{eq:severalupdates1:Dfastlambda} is precisely
\eqref{eq:thpfaffext:induction-residual-derivative}:
\[
        \|D(F_k^\ast\lambda)\|_{L^\infty(\B^n)}
        \le
        \Gamma_k^2.
\]

\item The $\Upsilon$ assumption \eqref{eq:severalupdates1:F2mlinftyupsilonasd} is precisely
\eqref{eq:thpfaffext:induction-Upsilon}:
\[
        \max_{1\leq m\leq d}
        \|F_k^{2m}\|_{L^\infty(\B^n)}
        \le
        \Upsilon_\ast.
\]
\end{itemize}

Thus all assumptions of \Cref{pr:severalupdates1} are satisfied and we obtain a map $F_{k+1}:=G$ be the map produced by the proposition. We now verify the induction hypotheses at level $k+1$.
\begin{itemize}
\item Since \Cref{pr:severalupdates1} gives $F_{k+1}\in C^\infty(\B^n;\R^N)$, we get
\eqref{eq:thpfaffext:induction-smooth} at level $k+1$.

\item By \eqref{eq:severalupdates1:outputsupp},
\[
        \supp(F_{k+1}-F_k)
        \subset
        \{x\in \B^n:\dist(x,K_k)<2\ell_k\}.
\]
Since $K_k\subset\Omega_k$ and
$2\ell_k<\dist(\Omega_k,\B^n\setminus\Omega_{k+1})$ by
\eqref{eq:thpfaffextv2:gap-condition}, we have
\[
        \supp(F_{k+1}-F_k)\subset\Omega_{k+1}.
\]
If $x\notin\Omega_{k+1}$, then $F_{k+1}=F_k$ near $x$, and also
$x\notin\Omega_k$. By \eqref{eq:thpfaffext:induction-support},
$F_k^\ast\lambda=0$ near $x$. Hence $F_{k+1}^\ast\lambda=0$ near $x$, and so
\[
        \supp F_{k+1}^\ast\lambda\subset\Omega_{k+1}.
\]
This proves \eqref{eq:thpfaffext:induction-support} at level $k+1$.

\item By \eqref{eq:severalupdates1:outputGastlambda} and
\eqref{eq:thpfaffext:output-compasdeps0askldj},
\[
        \|F_{k+1}^\ast\lambda\|_{L^\infty(\B^n)}
        \le
        C_{\mathrm{st}}\frac{r_k}{M_k}
        \le
        \eps_0 r_{k+1}.
\]
This proves \eqref{eq:thpfaffext:induction-residual} at level $k+1$.

\item By \eqref{eq:severalupdates1:outputDG2nNsGanna} and
\eqref{eq:thpfaffext:output-first-compatible},
\[
        \max_{1\leq m\leq d}
        \|D(F_{k+1}^{2m},F_{k+1}^{2m+1})\|_{L^\infty(\B^n)}
        \le
        C_{\mathrm{st}}M_k^s\Gamma_k
        \le
        \Gamma_{k+1}.
\]
This proves \eqref{eq:thpfaffext:induction-first} at level $k+1$.

\item By \eqref{eq:severalupdates1:outputD2G2nNsGanna} and
\eqref{eq:thpfaffext:output-second-compatible},
\[
        \max_{1\leq m\leq d}
        \|D^2(F_{k+1}^{2m},F_{k+1}^{2m+1})\|_{L^\infty(\B^n)}
        \le
        C_{\mathrm{st}}\frac{M_k^{3s}\Gamma_k^3}{r_k}
        \le
        \frac{\Gamma_{k+1}^3}{r_{k+1}}.
\]
This proves \eqref{eq:thpfaffext:induction-second} at level $k+1$.
\item By \eqref{eq:severalupdates1:outputDGastlambdasad33e} and
\eqref{eq:thpfaffext:output-residual-derivative-compatible},
\[
        \|D(F_{k+1}^\ast\lambda)\|_{L^\infty(\B^n)}
        \le
        C_{\mathrm{st}}M_k^{2s}\Gamma_k^2
        \le
        \Gamma_{k+1}^2.
\]
This proves \eqref{eq:thpfaffext:induction-residual-derivative} at level $k+1$.

\item By \eqref{eq:severalupdates1:outputGmF},
\[      \|F_{k+1}-F_k\|_{L^\infty(\B^n)}
        \le
        C_{\mathrm{st}}(1+\Upsilon_\ast)
        \frac{r_k}{M_k\Gamma_k}.
\]
which is \eqref{eq:thpfaffext:increment-C0} for $k+1$.

Also, by \eqref{eq:severalupdates1:outputDGmF}, for every $k\ge0$,

\[        \|D(F_{k+1}-F_k)\|_{L^\infty(\B^n)}
        \le
        C_{\mathrm{st}}(1+\Upsilon_\ast)M_k^s\Gamma_k.
\]
which is \eqref{eq:thpfaffext:increment-C1}

\item By the already established \eqref{eq:thpfaffext:increment-C0}, using \eqref{eq:thpfaffextv2:C0-summability},
\begin{equation*}
\begin{split}
        \max_{1\leq m\leq d}
        \|F_{k+1}^{2m}\|_{L^\infty(\B^n)}
        &\le
        \max_{1\leq m\leq d}
        \|F_0^{2m}\|_{L^\infty(\B^n)}
        +
        \sum_{i=0}^k
        \|F_{i+1}-F_i\|_{L^\infty(\B^n)}
        \\
        &\le
        \max_{1\leq m\leq d}
        \|F_0^{2m}\|_{L^\infty(\B^n)}
        +
        1
        \overset{\eqref{eq:thpfaffextv2:Upsilon-star}}{=}
        \Upsilon_\ast.
\end{split}
\end{equation*}
This proves \eqref{eq:thpfaffext:induction-Upsilon} at level $k+1$.
%
\end{itemize}

The induction is complete, and we have found a sequence $F_k \in C^\infty(\B^n)$ as claimed.

By interpolation, for any $\alpha \in (0,1]$
\begin{equation*}
\begin{split}
        \|F_{k+1}-F_k\|_{C^\alpha(\B^n)}
        \aleq&
        \|F_{k+1}-F_k\|_{L^\infty(\B^n)}
        +
        \|F_{k+1}-F_k\|_{L^\infty(\B^n)}^{1-\alpha}
        \|D(F_{k+1}-F_k)\|_{L^\infty(\B^n)}^\alpha\\
\overset{\eqref{eq:thpfaffext:increment-C0},\eqref{eq:thpfaffext:increment-C1}}{\aleq_{\Upsilon^\ast,C_{\mathrm{st}}}}&
\frac{r_k}{M_k\Gamma_k} + \brac{\frac{r_k}{M_k\Gamma_k}}^{1-\alpha} \brac{M_{k}^s\Gamma_{k}}^\alpha\\
=&\frac{r_k}{M_k\Gamma_k}
        +
        r_k^{1-\alpha}
        M_k^{(s+1)\alpha-1}
        \Gamma_k^{2\alpha-1}\\
        \overset{\eqref{eq:thpfaffextv2:rk-Gammak-relation}}{\leq}&
\frac{r_k}{M_k\Gamma_k}
        +
        \brac{r_0\Gamma_0^\sigma\Gamma_k^{-\sigma}}^{1-\alpha}
        \Gamma_k^{\gamma\brac{(s+1)\alpha-1}}
        \Gamma_k^{2\alpha-1}\\
       =& \frac{r_k}{M_k\Gamma_k}
        +
        (r_0\Gamma_0^\sigma)^{1-\alpha}
        \Gamma_k^{
        2\alpha-1-\sigma(1-\alpha)
        +
        \gamma((s+1)\alpha-1)
        }
        \\
        \overset{\eqref{eq:thpfaffext:gamma-choice}}{=}&
        \frac{r_k}{M_k\Gamma_k}
        +
        (r_0\Gamma_0^\sigma)^{1-\alpha}
        \Gamma_k^{{-\tau_\alpha}}.
        \end{split}
\end{equation*}
Observe that we can ensure $\tau_\alpha > 0$ in \eqref{eq:thpfaffext:gamma-choice} only since $\alpha$ satisfies \eqref{eq:alphasdoiuvc}.

By \eqref{eq:thpfaffextv2:C0-summability} and \eqref{eq:thpfaffext:Calpha-ratiosummable} for this $\alpha$ we have

\begin{equation}\label{eq:thpfaffext:Calpha-summability}
        \sum_{k=0}^\infty
        \|F_{k+1}-F_k\|_{C^\alpha(\B^n)}
        <\infty.
\end{equation}

We now record the same estimate with exponent $\beta$. By interpolation and
\eqref{eq:thpfaffext:increment-C0}--\eqref{eq:thpfaffext:increment-C1},
\begin{equation}\label{eq:thpfaffext:Cbeta-increment}
\begin{split}
        \|F_{k+1}-F_k\|_{C^\beta(\B^n)}
        \le
        \Sigma_\beta(1+\Upsilon_\ast)
        \left(
        \frac{r_k}{M_k\Gamma_k}
        +
        r_k^{1-\beta}
        M_k^{(s+1)\beta-1}
        \Gamma_k^{2\beta-1}
        \right).
\end{split}
\end{equation}
Using \eqref{eq:thpfaffextv2:rk-Gammak-relation} and $M_k=\Gamma_k^\gamma$,
the second term satisfies
\begin{equation}\label{eq:thpfaffext:Cbeta-second-term}
\begin{split}
        r_k^{1-\beta}
        M_k^{(s+1)\beta-1}
        \Gamma_k^{2\beta-1}
        &\le
        (r_0\Gamma_0^\sigma)^{1-\beta}
        \Gamma_k^{2\beta-1-\sigma(1-\beta)
        +\gamma((s+1)\beta-1)}
        \\
        &=
        (r_0\Gamma_0^\sigma)^{1-\beta}
        \Gamma_k^{-\tau_\beta}.
\end{split}
\end{equation}
Therefore, by
\eqref{eq:thpfaffextv2:C0-summability},
\eqref{eq:thpfaffext:Cbeta-ratiosummable}, and
\eqref{eq:thpfaffext:mu-beta-choice},
\begin{equation}\label{eq:thpfaffext:Cbeta-summability-small}
\begin{split}
        \sum_{k=0}^\infty
        \|F_{k+1}-F_k\|_{C^\beta(\B^n)}
        &\le
        \Sigma_\beta(1+\Upsilon_\ast)
        \left(
        2\frac{r_0}{\Gamma_0^{1+\gamma}}
        +
        2(r_0\Gamma_0^\sigma)^{1-\beta}\Gamma_0^{-\tau_\beta}
        \right)
        \\
        &=
        \Sigma_\beta(1+\Upsilon_\ast)
        \left(
        2\frac{r_0}{\Gamma_0^{1+\gamma}}
        +
        2r_0^{1-\beta}\Gamma_0^{-\mu_\beta}
        \right)
        \\
        &\overset{\mu_\beta >0}{<}
        \eps .
\end{split}
\end{equation}
Here the last inequality follows from the last two $\eps$-dependent conditions
in the definition of $\Gamma_0$, \eqref{eq:thpfaffextv2:Gamma0-choice}. And observe that $\mu_\beta > 0$ (whereas $\mu_\alpha < 0$, and the same inequality would fail for $\alpha$).

By \eqref{eq:thpfaffext:Calpha-summability} $F_k$ converges in $C^\alpha(\B^n;\R^N)$ to some $F\in C^\alpha(\B^n;\R^N)$. On the other hand by \eqref{eq:thpfaffext:induction-residual} and $r_k\to0$, and thus $\|F_k^\ast\lambda\|_{L^\infty(\B^n)}
        \to0$

Since $\alpha > \frac{1}{2}$ we conclude by \Cref{la:distrpullbackvanishes} that $F^\ast \lambda = 0$ in distributional sense.

Moreover since $\supp \brac{F_{k+1}-F_k} \subset \Omega_{k+1}$, we have $F(x)=F_0(x)$ whenever
$\dist(x,\partial \B^n)\le2\delta$. In particular, \eqref{eq:pfaffext:horizontalboundaryFF0} holds.

Since $F_k\to F$ in $C^\alpha(\B^n;\R^N)$, and hence in
$C^\beta(\B^n;\R^N)$, \eqref{eq:thpfaffext:Cbeta-summability-small} gives \eqref{eq:FF0closeby}
\end{proof}

\section{Making maps horizontal: Proof of Theorem~\ref{th:closebymaps}}

\label{sec:homeomorphic-graph-pfaff-nash}

\Cref{th:closebymaps} is a consequence of \Cref{th:pfaffconstantrank} and the following

\begin{theorem}
\label{th:homeopfaffapproxaslkd}
Let $N\ge 2d+1$,$d\ge1$,$n\ge2$, and recall $\lambda \in C^\infty(\Ep^1 \R^N)$ from \eqref{eq:pfaffcr}.
Assume that $b \in \{1,\ldots,d\}$ and set $s:=\left\lceil\frac{n+1}{b}\right\rceil$.

Let $F_0\in C^\infty(\overline{\B^n};\R^N)$.
Then for any $\Omega \subset \B^n$, any
\begin{equation}\label{eq:homeopfaff-alpha-range}
        \frac12<\alpha<\frac{s+1}{2s+1}.
\end{equation}
there exists $F\in C^\alpha(\B^n;\R^N)$ such that
\begin{equation}\label{eq:homeopfaff-output-horizontal}
        F^\ast\lambda=0 \quad \text{in $\mathcal{D}'(\Omega)$},
\end{equation}
and
\begin{equation}\label{eq:homeopfaff-unchangedcoordinates}
        F^{\ell} = F_0^\ell \quad \forall \ell \in \{2b+2,\ldots N\}
\end{equation}

Moreover, for every $0<\beta<1/2$ and every $\varepsilon>0$, the map
$F$ may be chosen so that
\begin{equation}\label{eq:homeopfaff-Cbeta-close}
        \|F-F_0\|_{C^\beta(\B^n)}<\varepsilon .
\end{equation}
\end{theorem}

\begin{proof}[Proof of \Cref{th:homeopfaffapproxaslkd}]

Fix $0<\beta<1/2$ and $\varepsilon>0$. As we did in \eqref{eq:thpfaffext:sigma-choice}, by \eqref{eq:homeopfaff-alpha-range} we can find $\sigma>0$ such that
\begin{equation}\label{eq:homeopfaff-sigma-choice}
        \frac{2\alpha-1}{1-\alpha}
        <
        \sigma
        <
        \frac1s .
\end{equation}

As in \eqref{eq:thpfaffext:gamma-choice}, \eqref{eq:thpfaffext:gamma-choice-beta}, \eqref{eq:thpfaffext:mu-beta-choice} we choose a small $\gamma>0$ sufficiently so that
\begin{equation}\label{eq:homeopfaff-tau-alpha-choice}
        \tau_\alpha
        :=
        \sigma(1-\alpha)-(2\alpha-1)
        -
        \gamma\bigl((s+1)\alpha-1\bigr)
        >0,
\end{equation}
\begin{equation}\label{eq:homeopfaff-tau-beta-choice}
        \tau_\beta
        :=
        \sigma(1-\beta)-(2\beta-1)
        -
        \gamma\bigl((s+1)\beta-1\bigr)
        >0,
\end{equation}
and
\begin{equation}\label{eq:homeopfaff-mu-beta-choice}
        \mu_\beta
        :=
        1-2\beta
        -
        \gamma\bigl((s+1)\beta-1\bigr)
        >0.
\end{equation}
The last condition uses $\beta<1/2$.

Let $\eps_{0,b}>0$ and $C_{\mathrm{st},b}\ge1$ be the constants in
\Cref{pr:keepcoefficientsargnobd}. Let $\Sigma_\beta\ge1$ be a constant
large enough for the $C^\beta$ interpolation estimate \eqref{eq:homeopfaffCbetainc} below. Set
\begin{equation}\label{eq:homeopfaff-C0-choice}
        C_0
        :=
        \max\left\{
        2,\,
        \frac{C_{\mathrm{st},b}}{\eps_{0,b}},\,
        C_{\mathrm{st},b},\,
        C_{\mathrm{st},b}^{1/2}
        \right\}.
\end{equation}

Set (this is the same as in \eqref{eq:thpfaffextv2:Upsilon-star}, just with $b$ instead of $d$)
\begin{equation}\label{eq:homeopfaff-Upsilon-star}
        \Upsilon_\ast
        :=
        1+
        \max_{1\leq m\leq b}
        \|F_0^{2m}\|_{L^\infty(\B^n)}
\end{equation}
and \eqref{eq:thpfaffextv2:r0-choice}
\[
        r_0
        :=
        \max\left\{
        1,\,
        \eps_{0,b}^{-1}
        \|F_0^\ast\lambda\|_{L^\infty(\B^n)}
        \right\}.
\]
Choose $\Gamma_0\ge1$ so large that (cf. \eqref{eq:thpfaffextv2:Gamma0-choice} with $d$ replaced by $b$)

\begin{equation}\label{eq:homeopfaff-Gamma0-initial}
\begin{split}
        \Gamma_0
        :=
        \max\Bigg\{&
        1,\,
        \max_{1\leq m\leq b}
        \|D(F_0^{2m},F_0^{2m+1})\|_{L^\infty(\B^n)},\\
        &\left(
        r_0
        \max_{1\leq m\leq d}
        \|D^2(F_0^{2m},F_0^{2m+1})\|_{L^\infty(\B^n)}
        \right)^{1/3},\\
        &\|D(F_0^\ast\lambda)\|_{L^\infty(\B^n)}^{1/2},\,
        r_0^{1/(1+\gamma)},\,
        (2C_0)^{1/\gamma},\\
        &\brac{2 C_0^{-(1+\gamma)} }^{\frac{1}{\gamma{1+s(\gamma+2)}}},\brac{2 C_0^{-\gamma}}^{\frac{1}{\gamma(1+s(\gamma+1))}}\\
        &C_0^{(1+\sigma)/(\gamma(1-s\sigma))},\,
        \left(\frac{8r_0}{1}\right)^{1/(\gamma+2)},\\
        &2^{1/(\gamma(1+s(\gamma+2)))},\,
        2^{1/(\gamma(1+s(\gamma+1)))},\,
        2^{1/(\tau_\alpha s\gamma)},\\
        &\bigl(2C_{\mathrm{st}}(1+\Upsilon_\ast)r_0\bigr)^{1/(\gamma+1)}\\,
        &2^{1/(\tau_\beta s\gamma)},
        \left(
        \frac{4\Sigma_\beta(1+\Upsilon_\ast)r_0}{\eps}
        \right)^{1/(1+\gamma)},
        \left(
        \frac{4\Sigma_\beta(1+\Upsilon_\ast)r_0^{1-\beta}}{\eps}
        \right)^{1/\mu_\beta},2^{1/(\tau_\alpha s\gamma)}
        \Bigg\}.
\end{split}
\end{equation}
Here, once more, $\Sigma_\beta \ge1$ is a constant, depending only on
$n,d,N,\beta,C_{\mathrm{st}}$, large enough for the $C^\beta$ interpolation
estimate in \eqref{eq:homeopfaffCbetainc} below.

For $k=0,1,2,\ldots$, and our definition of $s$, we define (as in  \eqref{eq:thpfaffextv2:parameters})
\[
        M_k:=\Gamma_k^\gamma,
        \qquad
        r_{k+1}:=C_0\frac{r_k}{M_k},
        \qquad
        \Gamma_{k+1}:=C_0M_k^s\Gamma_k, \qquad \Lambda_k:=\frac{\Gamma_k^3}{r_k}.
\]
We do not need $\ell_k$ since we do not need to respect any boundary data.

The parameter estimates are the same as in the proof of
\Cref{th:pfaffextension}, just with the new value of $s$.

In particular, the computation of \eqref{eq:thpfaffextv2:rk-Gammak-relation} gives
\begin{equation}\label{eq:homeopfaff-rk-Gammak-relation}
        r_k\leq r_0\Gamma_0^\sigma\Gamma_k^{-\sigma}
        \qquad\text{for every }k\ge0.
\end{equation}
The computations of \eqref{eq:thpfaffext:rk-scale-one} and
\eqref{eq:thpfaffext:rk-scale-two} give
\begin{equation}\label{eq:homeopfaff-rk-scale}
        r_k\leq M_k\Gamma_k,
        \qquad
        r_k\leq M_k^2\Gamma_k^2
        \qquad\text{for every }k\ge0.
\end{equation}
The ratio estimate \eqref{eq:thpfaffextv2:C0-summability} gives
\begin{equation}\label{eq:homeopfaff-C0-summability}
        \sum_{k=0}^{\infty}
        \frac{r_k}{M_k\Gamma_k}
        \le
        2\frac{r_0}{M_0\Gamma_0}
        =
        2\frac{r_0}{\Gamma_0^{1+\gamma}}.
\end{equation}
Similarly, using \eqref{eq:homeopfaff-Gamma0-initial}, the computation
of \eqref{eq:thpfaffext:Calpha-ratiosummable},\eqref{eq:thpfaffext:Cbeta-ratiosummable} give
\begin{equation}\label{eq:homeopfaff-alpha-ratio-summability}
        \sum_{k=0}^{\infty}\Gamma_k^{-\tau_\alpha}
        \le
        2\Gamma_0^{-\tau_\alpha},
\end{equation}
and
\begin{equation}\label{eq:homeopfaff-beta-ratio-summability}
        \sum_{k=0}^{\infty}\Gamma_k^{-\tau_\beta}
        \le
        2\Gamma_0^{-\tau_\beta}.
\end{equation}

Arguing as in
\eqref{eq:thpfaffext:output-compasdeps0askldj}--
\eqref{eq:thpfaffext:output-residual-derivative-compatible}, with
$C_{\mathrm{st}}$ replaced by $C_{\mathrm{st},b}$ and $\eps_0$ replaced
by $\eps_{0,b}$ we once more obtain,
\begin{equation}\label{eq:homeopfaff-output-residual-compatible}
        C_{\mathrm{st},b}\frac{r_k}{M_k}
        \le
        \eps_{0,b}r_{k+1},
\end{equation}
\begin{equation}\label{eq:homeopfaff-output-first-compatible}
        C_{\mathrm{st},b}M_k^s\Gamma_k
        \le
        \Gamma_{k+1},
\end{equation}
\begin{equation}\label{eq:homeopfaff-output-second-compatible}
        C_{\mathrm{st},b}
        \frac{M_k^{3s}\Gamma_k^3}{r_k}
        \le
        \frac{\Gamma_{k+1}^3}{r_{k+1}},
\end{equation}
and
\begin{equation}\label{eq:homeopfaff-output-residual-derivative-compatible}
        C_{\mathrm{st},b}M_k^{2s}\Gamma_k^2
        \le
        \Gamma_{k+1}^2.
\end{equation}

We now construct $F_k$ by induction.

The induction hypotheses are the same as in the proof of \Cref{th:pfaffextension}, except that we don't care about \eqref{eq:thpfaffext:induction-support}:

\begin{itemize}
\item The analogue of \eqref{eq:thpfaffext:induction-smooth}, i.e.
\[
        F_k\in C^\infty(\overline{\B^n};\R^N),
\]
\item The analogue of \eqref{eq:thpfaffext:induction-residual}, i.e. \begin{equation}\label{eq:homeopfaff-induction-residual}
        \|F_k^\ast\lambda\|_{L^\infty(\B^n)}
        \le
        \eps_{0,b}r_k,
\end{equation}

\item The analogue of \eqref{eq:thpfaffext:induction-first}, but with $b$ instead of $d$, i.e.
\begin{equation}\label{eq:homeopfaff-induction-first}
        \max_{1\leq m\leq b}
        \|D(F_k^{2m},F_k^{2m+1})\|_{L^\infty(\B^n)}
        \le
        \Gamma_k,
\end{equation}

\item The analogue of \eqref{eq:thpfaffext:induction-second} except with $b$ instead of $d$, \begin{equation}\label{eq:homeopfaff-induction-second}
        \max_{1\leq m\leq b}
        \|D^2(F_k^{2m},F_k^{2m+1})\|_{L^\infty(\B^n)}
        \le
        \frac{\Gamma_k^3}{r_k}=\Lambda_k
\end{equation}

\item The analogue of \eqref{eq:thpfaffext:induction-residual-derivative} \begin{equation}\label{eq:homeopfaff-induction-residual-derivative}
        \|D(F_k^\ast\lambda)\|_{L^\infty(\B^n)}
        \le
        \Gamma_k^2,
\end{equation}

\item The analogue of \eqref{eq:thpfaffext:induction-Upsilon} with $d$ replaced by $b$, \begin{equation}\label{eq:homeopfaff-induction-Upsilon}
        \max_{1\leq m\leq b}
        \|F_k^{2m}\|_{L^\infty(\B^n)}
        \le
        \Upsilon_\ast,
\end{equation}
\item \eqref{eq:thpfaffext:increment-C0} and \eqref{eq:thpfaffext:increment-C1}, i.e.
\begin{equation}\label{eq:homeopfaff-increment-C0}
        \|F_{k}-F_{k-1}\|_{L^\infty(\B^n)}
        \le
        C_{\mathrm{st},b}(1+\Upsilon_\ast)
        \frac{r_{k-1}}{M_{k-1}\Gamma_{k-1}},
\end{equation}
and
\begin{equation}\label{eq:homeopfaff-increment-C1}
        \|D(F_{k}-F_{k-1})\|_{L^\infty(\B^n)}
        \le
        C_{\mathrm{st},b}(1+\Upsilon_\ast)M_{k-1}^s\Gamma_{k-1}.
\end{equation}

\item and as a new condition that did hold in \Cref{th:pfaffextension} for $j=2d+2,\ldots,N$, but was irrelevant for us, but now becomes crucial:
\begin{equation}\label{eq:homeopfaff-induction-frozen}
        F_k^j=F_0^j
        \qquad\text{for every }j=2b+2,\ldots,N.
\end{equation}
\end{itemize}
Just as in the proof of \Cref{th:pfaffextension}, for $k=0$, these properties follow from the initial assumptions,
\eqref{eq:homeopfaff-Upsilon-star},
\eqref{eq:thpfaffextv2:r0-choice},
and \eqref{eq:homeopfaff-Gamma0-initial}.

Assume that $F_k$ has been constructed. Apply
\Cref{pr:keepcoefficientsargnobd} with
\[
        F=F_k,
        \qquad
        r=r_k,
        \qquad
        M=M_k,
        \qquad
        \Gamma=\Gamma_k,
        \qquad
        \Upsilon=\Upsilon_\ast.
\]
and all assumptions are satisfied by the induction hypothesis. We obtain some $G$ and set $F_{k+1}:=G$. By \eqref{eq:homeopfaff-output-residual-compatible}--\eqref{eq:homeopfaff-output-residual-derivative-compatible}, the induction closes just as in \Cref{th:pfaffextension}, the new condition \eqref{eq:homeopfaff-induction-frozen} is just \eqref{eq:vs2propassaslkdj-output-frozen}.

Once more, \eqref{eq:homeopfaff-C0-summability} and
\eqref{eq:homeopfaff-Gamma0-initial}, we have
\[
\begin{split}
        \sum_{i=0}^{\infty}
        \|F_{i+1}-F_i\|_{L^\infty(\B^n)}
        &\le
        C_{\mathrm{st},b}(1+\Upsilon_\ast)
        \sum_{i=0}^{\infty}\frac{r_i}{M_i\Gamma_i}
        \\
        &\le
        2C_{\mathrm{st},b}(1+\Upsilon_\ast)
        \frac{r_0}{\Gamma_0^{1+\gamma}}
        \le
        1.
\end{split}
\]

Convergence is now the verbatim to the proof of \Cref{th:pfaffextension}: Interpolating
\eqref{eq:homeopfaff-increment-C0} and \eqref{eq:homeopfaff-increment-C1}, and
repeating the computation leading to \eqref{eq:thpfaffext:Calpha-summability},
gives
\[
\begin{split}
        \|F_{k+1}-F_k\|_{C^\alpha(\B^n)}
        &\le
        C(1+\Upsilon_\ast)
        \left(
        \frac{r_k}{M_k\Gamma_k}
        +
        r_k^{1-\alpha}
        M_k^{(s+1)\alpha-1}
        \Gamma_k^{2\alpha-1}
        \right)
        \\
        &\le
        C(1+\Upsilon_\ast)
        \left(
        \frac{r_k}{M_k\Gamma_k}
        +
        (r_0\Gamma_0^\sigma)^{1-\alpha}
        \Gamma_k^{-\tau_\alpha}
        \right).
\end{split}
\]
Here we used \eqref{eq:homeopfaff-rk-Gammak-relation} and
\eqref{eq:homeopfaff-tau-alpha-choice}. By
\eqref{eq:homeopfaff-C0-summability} and
\eqref{eq:homeopfaff-alpha-ratio-summability},
\begin{equation}\label{eq:homeopfaff-Calpha-summability}
        \sum_{k=0}^{\infty}
        \|F_{k+1}-F_k\|_{C^\alpha(\B^n)}
        <\infty.
\end{equation}
Thus $F_k$ converges in $C^\alpha(\B^n;\R^N)$ to a map
\[
        F\in C^\alpha(\overline{\B^n};\R^N).
\]

Also, by \eqref{eq:homeopfaff-induction-residual} and $r_k\to0$,
\[
        \|F_k^\ast\lambda\|_{L^\infty(\B^n)}\to0.
\]
Since $\alpha>1/2$, \Cref{la:distrpullbackvanishes} gives
\[
        F^\ast\lambda=0
\]
in the distributional sense in the open set $\B^n$.

As for $C^\beta$-smallness, by the interpolation (with constant $\Sigma_\beta$) estimate,
using \eqref{eq:homeopfaff-increment-C0} and
\eqref{eq:homeopfaff-increment-C1}, we have
\begin{equation}\label{eq:homeopfaffCbetainc}
\begin{split}
        \|F_{k+1}-F_k\|_{C^\beta(\B^n)}
        \le
        \Sigma_\beta(1+\Upsilon_\ast)
        \left(
        \frac{r_k}{M_k\Gamma_k}
        +
        r_k^{1-\beta}
        M_k^{(s+1)\beta-1}
        \Gamma_k^{2\beta-1}
        \right).
\end{split}
\end{equation}
As in \eqref{eq:thpfaffext:Cbeta-second-term}, using
\eqref{eq:homeopfaff-rk-Gammak-relation} and
\eqref{eq:homeopfaff-tau-beta-choice}, the second term satisfies
\[
\begin{split}
        r_k^{1-\beta}
        M_k^{(s+1)\beta-1}
        \Gamma_k^{2\beta-1}
        &\le
        (r_0\Gamma_0^\sigma)^{1-\beta}
        \Gamma_k^{-\tau_\beta}.
\end{split}
\]
Therefore, by \eqref{eq:homeopfaff-C0-summability},
\eqref{eq:homeopfaff-beta-ratio-summability}, and
\eqref{eq:homeopfaff-mu-beta-choice},
\[
\begin{split}
        \sum_{k=0}^{\infty}
        \|F_{k+1}-F_k\|_{C^\beta(\B^n)}
        &\le
        \Sigma_\beta(1+\Upsilon_\ast)
        \left(
        2\frac{r_0}{\Gamma_0^{1+\gamma}}
        +
        2(r_0\Gamma_0^\sigma)^{1-\beta}
        \Gamma_0^{-\tau_\beta}
        \right)
        \\
        &=
        \Sigma_\beta(1+\Upsilon_\ast)
        \left(
        2\frac{r_0}{\Gamma_0^{1+\gamma}}
        +
        2r_0^{1-\beta}\Gamma_0^{-\mu_\beta}
        \right)
        <
        \varepsilon.
\end{split}
\]
The last inequality is \eqref{eq:homeopfaff-Gamma0-initial}. Since
$F_k\to F$ in $C^\beta(\B^n;\R^N)$, this gives
\[
        \|F-F_0\|_{C^\beta(\B^n;\R^N)}<\varepsilon.
\]

Lastly, \eqref{eq:homeopfaff-unchangedcoordinates} follows by passing to the uniform limit in
\eqref{eq:homeopfaff-induction-frozen}:
\[
        F^j=F_0^j
        \qquad\text{for every }j=2b+2,\ldots,N.
\]
The proof is complete.
\end{proof}

\section{Applications to the Heisenberg group: Proof of Corollaries~\ref{co:Heisenbergextension} and \ref{cor:heisenberghomeov1}}
\label{s:applicationstoheisenberg}

First we bring the Heisenberg contact form $\omega_{H_n}$ into the standard form $\lambda$ from \eqref{eq:pfaffcr}.

\begin{lemma}[Normalizing the Heisenberg contact form]
\label{la:heisenbergnormalied}
Let $\omega_{\H_n} \in C^\infty(\Ep^1 \R^{2n+1})$ from \eqref{eq:Hcontactform}.
Let us write it as
\[
        \omega_{\H_n}
        =
        dt+\frac{1}{2}\sum_{j=1}^n\bigl(x_j\,dy_j-y_j\,dx_j\bigr),
\]
be the Heisenberg contact form with coordinates $(x_1,y_1,\ldots,x_n,y_n,t) \in \R^n$

Let $\lambda$ be from \eqref{eq:pfaffcr}, i.e.
\[
        \lambda = dp^1+\sum_{j=1}^n p^{2j}\,dp^{2j+1}
\]
be the normalized Pfaff form on another copy of $\R^{2n+1}$ with coordinates
$p:=(p^1,\ldots,p^{2n+1})$.

Define $\Phi:\R^{2n+1}_{x,y,t}\to\R^{2n+1}_{p}$
\begin{equation}\label{eq:heisenberg-normalizing-coordinates}
        p^1 := \Phi^1(x,y,z)
        =
        t-\frac{1}{2}\sum_{j=1}^n x_jy_j,
        \qquad
        p^{2j} := \Phi^{2j}(x,y,z)
        =
        x_j,
        \qquad
        p^{2j+1}=\Phi^{2j+1}(x,y,z)
        =
        y_j,
\end{equation}
where $j=1,\ldots,n$. Moreover, set $\Psi:\R^{2n+1}_{p}\to \R^{2n+1}_{x,y,t}$
\begin{equation}\label{eq:heisenberg-normalizing-coordinates-inverse}
        x_j=\Psi^{2j-1}(p) :=p^{2j},
        \qquad
        y_j=\Psi^{2j}(p) :=p^{2j+1},
        \qquad
        t=\Psi^{2n+1}(p) := p^1+\frac12\sum_{j=1}^n p^{2j}p^{2j+1}.
\end{equation}
where $j=1,\ldots,n$.

Then $\Phi, \Psi: \R^{2n+1} \to \R^{2n+1}$ are global $C^\infty$ diffeomorphism, inverse to each other, $\Phi \circ \Psi(p) = p$, $\Psi\circ \Phi(x,y,t) = (x,y,t)$.

Moreover,
\begin{equation}\label{eq:heisenberg-contact-normal-form}
        \Phi^\ast\lambda=\omega_{\H_n}, \quad \text{and} \quad \Psi^\ast \omega_{\H_n} = \lambda.
\end{equation}
\end{lemma}
%

\begin{proof}[Proof of \Cref{co:Heisenbergextension}]
Let $\omega_{\H_n} \in C^\infty(\Ep^1 \R^{2n+1})$ be the Heisenberg contact form from \eqref{eq:Hcontactform}, and $\lambda\in C^\infty(\Ep^1 \R^{2n+1})$ the normalized form \eqref{eq:pfaffcr}. 

By \Cref{la:heisenbergnormalied}, there are global diffeomorphisms
$\Phi,\Psi:\R^{2n+1}\to\R^{2n+1}$, inverse to each other, such that
\[
        \Phi^\ast\lambda=\omega_{\H_n},
        \qquad
        \Psi^\ast\omega_{\H_n}=\lambda.
\]
Set
\[
        g:=\Phi\circ f
        \in C^\infty(\partial\B^k;\R^{2n+1}).
\]
Since $f$ is smooth and
$f\in\lip(\partial\B^k;\H_n)$, we have
$f^\ast\omega_{\H_n}=0$.  Hence
\[
        g^\ast\lambda
        =(\Phi\circ f)^\ast\lambda
        =f^\ast(\Phi^\ast\lambda)
        =f^\ast\omega_{\H_n}
        =0.
\]

Choose any smooth extension
\[
        G_0\in C^\infty(\overline{\B^k};\R^{2n+1}), \quad G_0 \Big |_{\partial \B^k} = g.
\]
By \eqref{eq:heisenberghomeo-alpha} we can apply \Cref{th:pfaffextension} with
\[
        N=2n+1,
        \qquad
        d=n,
        \qquad
        \lambda,
        \qquad
        \text{domain dimension } k,
        \qquad
        F_0=G_0 .
\]
and find $G\in C^\alpha(\overline{\B^k};\R^{2n+1})$ with $G \Big |_{\partial \B^{k}} = g$ and
$G^\ast\lambda=0$ in the distributional sense on $\overline{\B^k}$. Set $F := \Psi \circ G \equiv \Phi^{-1} \circ G \in C^\alpha(\overline{\B^k},\R^{2n+1})$. By  \Cref{la:diffeodistributionalpullback}
\[
 F^\ast \omega_{\H^n} = G^\ast \lambda = 0
\]
again in distributional sense in $\overline{\B^k}$.

By \cite[Theorem 7.4]{HajlaszMirraSchikorra}, a Euclidean
$C^\alpha$ map with $\alpha>1/2$ and vanishing distributional pullback of the
Heisenberg contact form is $C^\alpha$ as a map into $\mathbb H_n$. Therefore
\[
        F\in C^\alpha(\overline{\B^k};\mathbb H_n).
\]
This proves the extension statement.
\end{proof}

\Cref{cor:heisenberghomeov1} follows from
\begin{corollary}[H\"older embedding into $\mathbb H_n$]
\label{cor:heisenberghomeov2}
Let $n\ge3$ and set
\[
        b_n:=\left\lfloor\frac{n-1}{2}\right\rfloor,
        \qquad
        s_n:=\left\lceil\frac{n+2}{b_n}\right\rceil .
\]
Then, for every
\[
        \frac12<\alpha<\frac{s_n+1}{2s_n+1},
\]
there exists an embedding
\[
        F:\overline{\B^{n+1}}\to F(\overline{\B^{n+1}})\subset \mathbb H_n
\]
such that
\[
        F\in C^\alpha(\overline{\B^{n+1}};\R^{2n+1})
        \qquad\text{and}\qquad
        F\in C^\alpha_{\loc}(\B^{n+1};\mathbb H_n).
\]

In particular, the following dimensions and ranges are obtained:
\[
\begin{array}{c|c|c|c}
        n & b_n & s_n & \text{admissible range} \\
        \hline
        3 & 1 & 5 & \frac12<\alpha<\frac6{11} \\
        4 & 1 & 6 & \frac12<\alpha<\frac7{13} \\
        5 & 2 & 4 & \frac12<\alpha<\frac59 \\
        6 & 2 & 4 & \frac12<\alpha<\frac59 \\
        7 & 3 & 3 & \frac12<\alpha<\frac47 \\
        8 & 3 & 4 & \frac12<\alpha<\frac59 \\
        n\ge9 & \left\lfloor\frac{n-1}{2}\right\rfloor & 3
        & \frac12<\alpha<\frac47 .
\end{array}
\]
\end{corollary}

\begin{proof}
Let $\omega_{\H_n} \in C^\infty(\Ep^1 \R^{2n+1})$ be the Heisenberg contact form from \eqref{eq:Hcontactform}, and $\lambda\in C^\infty(\Ep^1 \R^{2n+1})$ the normalized form \eqref{eq:pfaffcr}. By
\Cref{la:heisenbergnormalied}, there are global diffeomorphism $\Phi, \Psi: \R^{2n+1} \to \R^{2n+1}$ inverse to each other with
\[
        \Phi^\ast\omega_{\H_n}=\lambda, \quad \Psi^\ast \lambda = \omega_{\H_n}.
        \]

Choose any smooth initial map $G_0:\overline{\B^{n+1}}\to\R^{2n+1}$ such that
\[
        G_0^{2b_n+1+i}(x)=x^i,
        \qquad
        i=1,\ldots,n+1,
\]
This is possible since $2b_n+n+2\leq 2n+1$, i.e. $b_n \leq \frac{n-1}{2}$, and $b_n \geq 1$ since $n \geq 3$.

Apply \Cref{th:homeopfaffapproxaslkd} to $G_0$ -- observe that $d=n$, $N=2n+1$, but ``$n$'' is our $n+1$, so
\[
 b := \min\left \{n,\left \lfloor \frac{2n+1-(n+1)-1}{2} \right \rfloor \right \}=b_n \geq 1
\]
Thus, we find a map, $G \in C^\alpha(\B^{n+1},\R^{2n+1})$,
\[
 G^\ast \lambda = 0 \quad \text{in $\mathcal{D}'(\frac{1}{2} \B^{n+1})$}
\]
Moreover, by \eqref{eq:homeopfaff-unchangedcoordinates}
\[
        G^{2b_n+1+i}(x)=G_0^{2b_n+1+i}(x)=x^i
        \qquad
        i=1,\ldots,n+1,
\]
Thus $G: \frac{1}{2} \B^{n+1} \to \R^{2n+1}$ is injective. Moreover,
\[
        G^{-1}(p)=\bigl(p^{2b_n+2},p^{2b_n+3},\ldots,p^{2b_n+n+2}\bigr) \quad \forall p = (p^1,\ldots,p^{2n+1}) \in G(\B^{n+1})
\]
Hence $G^{-1}: G(\B^{n+1}) \to \R^{n+1}$ is the restriction of a smooth map to $G(\B^{n+1})$ -- in particular $G: \B^{n+1} \to G(\B^{n+1})$ is a homeomorphism.

Now define $F:=\Psi\circ G$. Since $\Psi$ is a diffeomorphism, $F$ is a homeomorphism of
$\overline{\B^{n+1}}$ onto its image. Also $F\in C^\alpha(\overline{\B^{n+1}};\R^{2n+1})$. By \Cref{la:diffeodistributionalpullback} $F^\ast \lambda = 0$ $\frac{1}{2} \B^{n+1}$, in distributional sense.

Since $\alpha>1/2$, the
characterization of H\"older maps into the Heisenberg group by the
distributional contact equation, \cite[Theorem 7.4]{HajlaszMirraSchikorra},
implies $F\in C^\alpha(\frac{1}{2} \B^{n+1};\mathbb H_n)$. A simple rescaling then also finds an embedding $\B^{n+1} \to \H_n$.

It remains only to compute the displayed ranges. If $n=2q+1$, then
\[
        b_n=q,
        \qquad
        s_n=\left\lceil\frac{2q+3}{q}\right\rceil .
\]
Thus $s_n=5$ for $n=3$, $s_n=4$ for $n=5$, and $s_n=3$ for every odd
$n\ge7$. If $n=2q$, then
\[
        b_n=q-1,
        \qquad
        s_n=\left\lceil\frac{2q+2}{q-1}\right\rceil .
\]
Thus $s_n=6$ for $n=4$, $s_n=4$ for $n=6,8$, and $s_n=3$ for every even
$n\ge10$. We can conclude.
\end{proof}

\appendix
\section{The case $n=1$}

A natural approach to disproving Gromov's conjecture, \Cref{conj:gromov},
for $n=1$ is to adapt the corrugation profiles in \Cref{la:profile} so that
injectivity is preserved. Convex integration constructions with control of
injectivity are available, for example in \cite{FMCO18}, although the
laminate-based methods used there do not seem directly applicable to the
Pfaffian equation considered here. We explored this approach using ChatGPT Pro. After several unsuccessful attempts, on September 5, 2026, the author's
instance of GPT-6 Astra produced a candidate argument for such an
adaptation.

The proposed argument combines the main idea of this paper with techniques
from more recent works on convex integration, such as \cite{BDLS2013Transport,CaoHirschInauen2025,DeLellisInauenSzekelyhidi2018}

At the time of writing, the argument remains unverified. In collaboration
with P.~Haj\l{}asz, the author is examining, verifying, and improving its details; the current version
of the proposed argument is available in \cite{AstraHajlaszSchikorra}.

\bibliographystyle{abbrv}
\bibliography{bib}

\end{document}